\documentclass[11pt]{amsart}

\usepackage{esint,amssymb,geometry,verbatim, color}
\usepackage[colorlinks=true,urlcolor=blue, citecolor=red,linkcolor=blue,linktocpage,pdfpagelabels, bookmarksnumbered,bookmarksopen]{hyperref}
\usepackage[hyperpageref]{backref}

\numberwithin{equation}{section}

\newtheorem{prop}{Proposition}[section]
\newtheorem{lm}[prop]{Lemma}

\newtheorem{teo}[prop]{Theorem}

\theoremstyle{definition}
\newtheorem{oss}[prop]{Remark}
\newtheorem*{ack}{Acknowledgments}

\title[Orthotropic $p-$harmonic functions]{On the Lipschitz character of\\ orthotropic $p-$harmonic functions}
\author{P. Bousquet}
\author{L. Brasco}
\author{C. Leone}
\author{A. Verde}

\date{29/01/2018}

\address[P. Bousquet]{Institut de Math\'ematiques de Toulouse, CNRS UMR 5219
\newline\indent Universit\'e de Toulouse
\newline\indent F-31062 Toulouse Cedex 9, France.}
\email{pierre.bousquet@math.univ-toulouse.fr}

\address[L.\ Brasco]{Dipartimento di Matematica e Informatica
	\newline\indent
	Universit\`a degli Studi di Ferrara
	\newline\indent
	Via Machiavelli 35, 44121 Ferrara, Italy}
\address{{\it and }
	Aix-Marseille Universit\'e, CNRS
	\newline\indent
	Centrale Marseille, I2M, UMR 7373, 39 Rue Fr\'ed\'eric Joliot Curie
	\newline\indent
	13453 Marseille, France}
\email{lorenzo.brasco@unife.it}

\address[C.\ Leone \& A.\ Verde]{Dipartimento di Matematica ``R. Caccioppoli''
\newline\indent
Universit\`a degli Studi di Napoli ``Federico II''
\newline\indent
Via Cinthia, Complesso Universitario di Monte S. Angelo, 80126 Napoli, Italy}
\email{chiara.leone@unina.it}
\email{anna.verde@unina.it}

\keywords{Degenerate elliptic equations; Anisotropic problems; Orthotropic problems; Lipschitz regularity}
\subjclass[2010]{35J70, 35B65, 49K20}

\begin{document}

\begin{abstract}
We prove that local weak solutions of the orthotropic $p-$harmonic equation are locally Lipschitz, for every $p\ge 2$ and in every dimension. More generally, the result holds true for more degenerate equations with orthotropic structure, with right-hand sides in suitable Sobolev spaces.
\end{abstract}

\maketitle 

\begin{center}
\begin{minipage}{11cm}
\small
\tableofcontents
\end{minipage}
\end{center}

\section{Introduction}

\subsection{The problem}

In this paper, we pursue the study of the regularity of local minimizers of degenerate {\it functionals with orthotropic structure}, that we already considered in \cite{BB, BBJ, BC} and \cite{BLPV}.
More precisely, for $p\ge 2$, we consider local minimizers of the functional
\begin{equation}
\label{puccetto}
\mathfrak{F}_0(u,\Omega')=\sum_{i=1}^N \frac{1}{p}\,\int_{\Omega'} |u_{x_i}|^p\,dx,\qquad \Omega'\Subset\Omega,\ u\in W^{1,p}_{\rm loc}(\Omega'),
\end{equation}
and more generally of the functional
\[
\mathfrak{F}_\delta(u,\Omega')=\sum_{i=1}^N \frac{1}{p}\,\int_{\Omega'} (|u_{x_i}|-\delta_i)_{+}^p\,dx +\int_{\Omega'}f\,u\,dx ,\qquad \Omega'\Subset\Omega,\ u\in W^{1,p}_{\rm loc}(\Omega').
\]
Here, \(\Omega\subset\mathbb{R}^N\) is an open set, \(N\geq 2\), and \(\delta_1, \dots, \delta_N\) are nonnegative numbers.
\par
A local minimizer $u$ of the functional $\mathfrak{F}_0$ defined in \eqref{puccetto} is a local weak solution of the {\it orthotropic $p-$Laplace} equation
\begin{equation}
\label{puccettona}
\sum_{i=1}^N \left(|u_{x_i}|^{p-2}\,u_{x_i}\right)_{x_i}=0.
\end{equation}
For $p=2$ this is just the Laplace equation, which is uniformly elliptic. For $p>2$ this looks quite similar to the usual $p-$Laplace equation
\[
\sum_{i=1}^N \left(|\nabla u|^{p-2}\,u_{x_i}\right)_{x_i}=0,
\]
whose local weak solutions are local minimizers of the functional
\begin{equation}
\label{silvani}
\mathfrak{I}(u,\Omega')=\frac{1}{p}\,\int_{\Omega'} |\nabla u|^p\,dx,\qquad \Omega'\Subset\Omega,\ u\in W^{1,p}_{\rm loc}(\Omega').
\end{equation}
However, as explained in  \cite{BB} and \cite{BBJ}, equation \eqref{puccettona} is much more degenerate. Consequently, as for the regularity of $\nabla u$ (i.e. boundedness and continuity), the two equations are dramatically different. 
\par
In order to understand this discrepancy between the $p-$Laplacian and its orthotropic version, let us observe that the map \(\xi \mapsto |\xi|^p\) occuring in the definition \eqref{silvani} of $\mathfrak{I}$ degenerates only at the origin, in the sense that its Hessian is positive definite on \(\mathbb{R}^N\setminus \{0\}\). On the contrary, the definition of the orthotropic functional $\mathfrak{F}_0$ in \eqref{puccetto} is related to the map \(\xi\mapsto \sum_{i=1}^N |\xi_i|^p\), which {\it degenerates on an unbounded} set, namely the $N$ hyperplanes orthogonal to the coordinate axes of \(\mathbb{R}^N\). 
\par
The situation is even worse when 
\begin{equation}
\label{deltai}
\max \{\delta_i\, :\, i=1,\dots,N\}>0,
\end{equation} 
for the lack of ellipticity of the \emph{degenerate} $p-$orthotropic functional  arises on the larger set 
\[
\bigcup_{i=1}^N \{\xi \in \mathbb{R}^N : |\xi_i|\leq \delta_i\}.
\] 
As a matter of fact, the regularity theory for these very degenerate functionals is far less understood than the corresponding theory for the standard case \eqref{silvani} and its variants. 
\par
Under suitable integrability conditions on the function \(f\), we can use the classical theory for functionals with $p-$growth and ensure that the local minimizers of $\mathfrak{F}_\delta$ are  locally bounded and H\"older continuous, see for example \cite[Theorems 7.5 \& 7.6]{Gi}. This theory also assures that the gradients of local minimizers lie in $L^r_{\rm loc}(\Omega)$ for some $r>p$, see \cite[Theorem 6.7]{Gi}. 
\par
We also point out that for $f\in L^\infty_{\rm loc}(\Omega)$, local minimizers of $\mathfrak{F}_\delta$ are contained in $W^{1,q}_{\rm loc}(\Omega)$, for every $q<+\infty$ (see \cite[Main Theorem]{BC}).
 
\subsection{Main result}
In this paper, we establish the optimal regularity expected for the minimizers of $\mathfrak{F}_\delta$, namely the \emph{Lipschitz regularity}\footnote{Observe  that when \(f\equiv 0\), any Lipschitz function $u$ with \(|\nabla u|\le \min\{\delta_i\, :\, i=1,\dots,N\}\) is a local minimizer of $\mathfrak{F}_\delta$. Thus in general Lipschitz continuity is the best regularity one can hope for.}. More precisely, we establish the following result.
\begin{teo}
\label{teo:lipschitz}
Let $p\ge 2$, $f\in W^{1,h}_{\rm loc}(\Omega)$ for some \(h>N/2\) and let $U\in W^{1,p}_{\rm loc}(\Omega)$ be a local minimizer of the functional $\mathfrak{F}_\delta$. Then $U$ is locally Lipschitz in $\Omega$.
\par
Moreover, in the case \(\delta_1=\dots=\delta_N=0\), we have the following local scaling invariant estimate: for every ball \(B_{2R_0}\Subset \Omega\), it holds
\begin{equation}
\label{stimayeah}
\|\nabla U\|_{L^\infty(B_{R_0/2})}\le C\, \left(\fint_{B_{R_0}} |\nabla U|^{p}\,dx\right)^\frac{1}{p}+C\,\left[R_0^2\,\left(\fint_{B_{R_0}} |\nabla f|^h\,dx\right)^\frac{1}{h}\right]^\frac{1}{p-1},
\end{equation}
for some $C=C(N,p,h)>1$.
\end{teo}
\begin{oss}[Comparison with previous results]
This result unifies and substantially extends the results on the orthotropic functional $\mathfrak{F}_\delta$ contained in  \cite{BBJ}, where it has been established that the  local minimizers of $\mathfrak{F}_\delta$ are locally Lipschitz, provided that: 
\begin{itemize}
\item $p\ge 2$, $N=2$ and $f\in W^{1,p'}_{\rm loc}(\Omega)$, see \cite[Theorem A]{BBJ};
\vskip.2cm
\item $p\ge 4$, $N\ge 2$ and $f\in W^{1,\infty}_{\rm loc}(\Omega)$, see \cite[Theorem B]{BBJ}.
\end{itemize}
The second result was based on the so-called {\it Bernstein's technique}, see for example \cite[Proposition 2.19]{HL}. This technique had already been exploited in the pioneering paper \cite{UU} by Uralt'seva and Urdaletova, for a class of functionals which contains the orthotropic functional $\mathfrak{F}_0$ defined in \eqref{puccetto}, but not its more degenerate version $\mathfrak{F}_\delta$. Namely, the result of \cite{UU} does not cover the case when condition \eqref{deltai} is in force. 
\par
Still for the case $\delta_1=\dots=\delta_N=0$, an entirely different approach relying on viscosity methods has been developped in \cite{D}. 
To our knowledge, both methods are limited to (at least) \emph{bounded} lower order terms $f$. 
\par
On the contrary, \cite[Theorem A]{BBJ} can be considered as the true ancestor to Theorem \ref{teo:lipschitz} above. Indeed, they both follow the {\it Moser's iteration technique}, originally introduced in \cite{Mo} to establish regularity for uniformly elliptic problems. However, going beyond the two-dimensional setting requires new ideas, that we will explain in Subsection \ref{ssec:abel} below. 
\end{oss}
In contrast to the partial results of \cite[Theorems A \& B]{BBJ}, the proof of Theorem \ref{teo:lipschitz} does not depend on the dimension and does not need any additional restriction on $p$, apart from $p\ge 2$. It allows unbounded lower order terms, even if the condition $f\in W^{1,h}_{\rm loc}(\Omega)$ for some \(h>N/2\) is certainly not sharp. On this point, it is useful to observe that by Sobolev's embedding we have\footnote{We recall that
\[
h^*=\left\{\begin{array}{rl}
N\,h/(N-h),& \mbox{ if }h<N,\\
\mbox{ any }q<+\infty, & \mbox{ if }h=N,\\
+\infty,& \mbox{ if }h>N.
\end{array}
\right.
\]}
\[
W^{1,h}\hookrightarrow L^{h^*},
\]
with $h^*$ larger than $N$ and as close to $N$ as desired, provided $h$ is close to $N/2$.
This means that, in terms of summability, our assumption on \(f\) amounts to \(f\in L^q_{\rm loc}(\Omega)\) for some \(q>N\). This is exactly the sharp expected condition on $f$ for the local minimizers to be locally Lipschitz, at least if one nurtures the (optimistic) hope that the regularity for the orthotropic $p-$Laplacian agrees with that for the standard $p-$Laplacian\footnote{In the case of the standard $p-$Laplacian, the sharp assumption to have Lipschitz regularity is $f\in L^{N,1}_{\rm loc}$, the latter being a Lorentz space. This sharp condition has been first detected by Duzaar and Mingione in \cite[Theorem 1.2]{DM}, see also \cite[Corollary 1.6]{KM} for a more general and refined result. This sharp result is obtained by using potential estimates techniques. We recall that $L^q_{\rm loc}\subset L^{N,1}_{\rm loc}$ for every $q>N$ and under this slightly stronger assumption on $f$, Lipschitz regularity for the $p-$Laplacian can be proved by more standard techniques based on Moser's iteration, see for example \cite{Br}.}.
\par 
Our strategy to prove Theorem \ref{teo:lipschitz} relies on energy methods and integral estimates, and more precisely on {\it ad hoc} Caccioppoli-type inequalities. This only requires growth assumptions on the Lagrangian and its derivatives and can be adapted to a large class of functionals. For instance, we briefly explain in Appendix \ref{sec:nllot} how to adapt our poof to the case of \emph{nonlinear} lower order terms, i.e. when \(f\, u\) is replaced by a term of the form \(G(x,u)\).

\begin{oss}
We collect in this remark some interesting open issues:
\begin{enumerate}
\item one word about the assumption \(p\geq 2\): as explained in \cite{BB} and \cite{BBJ}, when $\delta_1=\dots=\delta_N=0$ the subquadratic case $1<p<2$ is simpler in a sense. In this case, the desired Lipschitz regularity can be inferred from \cite[Theorem 2.2]{FF} (see also \cite[Theorem 2.7]{FFM}). However, the more degenerate case \eqref{deltai} is open;
\vskip.2cm
\item in \cite[Main Theorem]{BB}, local minimizers were proven to be $C^1$, in the two-dimensional case, for $1<p<\infty$ and when $\delta_1=\dots=\delta_N=0$. We also refer to the very recent paper \cite{LR}, where a modulus of continuity for the gradient of local mimizers is exhibited.
We do not know whether such a result still holds in higher dimensions;
\vskip.2cm
\item in \cite[Theorem 1.4]{BLPV}, local Lipschitz regularity is established in the two-dimensional case for an orthotropic functional, with anisotropic growth conditions; that is, for the functional
\[
\sum_{i=1}^2 \frac{1}{p_i}\,\int (|u_{x_i}|-\delta_i)_{+}^{p_i}\,dx +\int f\,u\,dx,\qquad \mbox{ with }2\le p_1\le p_2.
\]
For such a functional, Lipschitz regularity is open in higher dimensions, even for the case $\delta_1=\dots=\delta_N=0$, i.e. for the functional
\[
\sum_{i=1}^2 \frac{1}{p_i}\,\int |u_{x_i}|^{p_i}\,dx +\int f\,u\,dx,\qquad \mbox{ with }2\le p_1\le p_2\le\dots\le p_N.
\] 
We point out that in this case, Lipschitz regularity in every dimension has been obtained in \cite[Theorem 1]{UU} for {\it bounded} local minimizers, under the additional restrictions
\[
p_1\ge 4\qquad \mbox{ and }\qquad p_N<2\,p_1.
\]
Though these restrictions are not optimal, we recall that regularity can not be expected when $p_N$ and $p_1$ are too far part, due to the well-known counterexamples of Giaquinta \cite{Gi} and Marcellini \cite{Ma}.
\end{enumerate}
\end{oss}

\subsection{Technical novelties of the proof}
\label{ssec:abel}

Our main result is obtained by considering a regularized problem having a unique smooth solution converging to our local minimizer, and proving a local Lipschitz estimate independent of the regularization parameter. 
\par
At first sight, the strategy to prove such an estimate may seem quite standard: 
\begin{itemize}
\item[{\bf a)}] differentiate equation \eqref{puccettona}; 
\vskip.2cm
\item[{\bf b)}] obtain Caccioppoli-type inequalities for convex powers of the components $u_{x_k}$ of the gradient;
\vskip.2cm
\item[{\bf c)}] derive an iterative scheme of reverse H\"older's inequalities;
\vskip.2cm
\item[{\bf d)}] iterate and obtain the desired local $L^\infty$ estimate on $\nabla u$.
\end{itemize}
However, steps {\bf b)} and {\bf c)} are quite involved, due to the degeneracy of our equation.
This makes their concrete realization fairly intricate. Thus in order to smoothly introduce the reader to the proof, we prefer to spend some words.
\par
We point out that our proof is not just a mere adaption of techniques used for the $p-$Laplace equation. Moreover, it does not even rely on the ideas developed in \cite{BBJ} for the two-dimensional case. In a nutshell, we need new ideas to deal with our functional in full generality.
\vskip.2cm\noindent
In order to obtain ``good'' Caccioppoli-type inequalities for the gradient, we exploit an idea introduced in nuce in \cite{BB}. This consists in differentiating \eqref{puccettona} in the direction $x_j$ and then testing the resulting equation with a test function of the form\footnote{This test function is not really admissible, since it is not compactly supported. Actually, to make it admissible we have to multiply it by a cut-off function. However, this gives unessential modifications, we prefer to avoid it in order to neatly present the idea of the proof.}
\[
u_{x_j}|u_{x_j}|^{2s-2}\,|u_{x_k}|^{2m},
\]
with $1\le s\le m$. This leads to an estimate of the type (see Proposition \ref{stair})
\begin{equation}
\label{powerchiaraintro}
\begin{split}
\sum_{i=1}^N \int_{\Omega} |u_{x_i}|^{p-2}\,u_{x_i x_j}^2\,|u_{x_j}|^{2\,s-2}\,|u_{x_k}|^{2\,m}\,\,dx&\le C\, \sum_{i=1}^N \int |u_{x_i}|^{p-2}\,\left(|u_{x_j}|^{2\,s+2\,m}+|u_{x_k}|^{2\,s+2\,m}\right)\,dx\\
& + \sum_{i=1}^N \int |u_{x_i}|^{p-2}\,u_{x_i x_j}^2\,|u_{x_j}|^{4\,s-2}\,|u_{x_k}|^{2\,m-2\,s}\,dx.
\end{split}
\end{equation}
Then the idea is the following: let us suppose that we are interested in improving the summability of the component $u_{x_k}$. Ideally, we would like to take $s=1$ in \eqref{powerchiaraintro}, since in this case the left-hand side boils down to
\[
\begin{split}
\sum_{i=1}^N \int |u_{x_i}|^{p-2}\,u_{x_i x_j}^2\,u_{x_k}^{2\,m}\,dx&\ge \int |u_{x_k}|^{p-2}\,u_{x_k x_j}^2\,u_{x_k}^{2\,m}\,dx\\
&\simeq \int \left|\left(|u_{x_k}|^{\frac{p}{2}+m}\right)_{x_j}\right|^2\,dx.
\end{split}
\]
If we now sum over \(j=1,\dots,N\), this would give a control on the $W^{1,2}$ norms of convex powers of  $u_{x_k}$. {\it But there is a drawback here}: indeed, this $W^{1,2}$ norm is estimated still in terms of the Hessian of $u$, which is contained in the right-hand side of \eqref{powerchiaraintro}. Observe that \eqref{powerchiaraintro} has the following form
\begin{equation}
\label{powerchiaraintroI}
\mathcal{I}(s-1,m)\le C\,\sum_{i=1}^N \int |u_{x_i}|^{p-2}\,\left(|u_{x_j}|^{2\,s+2\,m}+|u_{x_k}|^{2\,s+2\,m}\right)\,dx+\mathcal{I}(2\,s-1,m-s),
\end{equation}
where
\[
\mathcal{I}(s,m)=\sum_{i=1}^N \int |u_{x_i}|^{p-2}\,u_{x_i x_j}^2\,|u_{x_j}|^{2\,s}\,|u_{x_k}|^{2\,m}\,\,dx.
\]
This suggests to perform a finite iteration of \eqref{powerchiaraintroI} for $s=s_i$ and $m=m_i$ such that
\[
\left\{\begin{array}{rcl} 
2\,s_i-1&=&s_{i+1}-1\\
s_0&=&1
\end{array}\right.\qquad \mbox{ and }\qquad m_i-s_i=m_{i+1},\qquad \mbox{ for } i=0,\dots,\ell. 
\]
The number $\ell$ is chosen so that we stop the iteration when we reach $m_\ell=0$. The above conditions imply that for every $i=0,\dots,\ell$, we have
\[
m_i+s_i=m_0+s_0=2^\ell.
\]
In this way, after a finite number of steps (comparable to $\ell$), the coupling between $u_{x_k}$ and the Hessian of $u$ contained in the term $\mathcal{I}$ will disappear from the right-hand side. In other words, we will end up with an estimate of the type 
\begin{equation}
\label{marò}
\begin{split}
\int \left|\nabla |u_{x_k}|^{2^\ell+\frac{p-2}{2}}\right|^2\,dx
&\le C\,\sum_{i,j=1}^N\int |u_{x_i}|^{p-2}\,\left(|u_{x_j}|^{2^{\ell+1}}+|u_{x_k}|^{2^{\ell+1}}\right)\,dx\\
&+\sum_{i=1}^N \int |u_{x_i}|^{p-2}\,u_{x_i x_j}^2\,u_{x_j}^{2\,(2^{\ell}-1)}\,dx.
\end{split}
\end{equation}
Observe that we still have the Hessian of $u$ in the right-hand side (this is the second term), but this time it is harmless. It is sufficient to use the standard Caccioppoli inequality \eqref{mothergsob} for the gradient, which reads
\[
\sum_{i=1}^N \int |u_{x_i}|^{p-2}\,u_{x_i x_j}^2\,u_{x_j}^{2\,(2^{\ell}-1)}\,dx\lesssim \sum_{i=1}^N \int |u_{x_i}|^{p-2}\,u_{x_j}^{2^{\ell+1}}\,dx,
\]
and the last term is already contained in the right-hand side of \eqref{marò}. All in all, by applying Sobolev inequality in the left-hand side of \eqref{marò}, we get the following type of self-improving information
\[
\nabla u\in L^{2\,\gamma}(B_R)\qquad \Longrightarrow\qquad \nabla u\in L^{2^*\gamma}(B_r),\qquad \mbox{ where we set } \gamma=\frac{p-2}{2}+2^\ell.
\]
In this way, we obtain an iterative scheme of reverse Holder's inequalities. This is {\bf Step 1} in the proof of Proposition \ref{prop:a_priori_estimate} below. Thus, apparently, we safely landed in step {\bf c)} of the strategy described above.
\vskip.2cm
We now want to pass to step {\bf d)} and iterate infinitely many times the previous information. The goal would be to define the diverging sequence of exponents $\gamma_\ell$ by
\[
\gamma_{\ell}=\frac{p-2}{2}+2^\ell,\qquad \ell\ge 1,
\]
and conclude by iterating
\begin{equation}
\label{groviera}
\nabla u\in L^{2\,\gamma_\ell}(B_R) \qquad \Longrightarrow \qquad \nabla u \in L^{2^*\gamma_\ell}(B_r).
\end{equation}
Once again, {\it there is a drawback}. Indeed, observe that by definition
\[
\frac{2^*}{2}\,\gamma_\ell\not= \gamma_{\ell+1}.
\]
  One may think that this is not a big issue: indeed, it would be sufficient to have 
\begin{equation}
\label{condizione0}
\gamma_{\ell+1}\le \frac{2^*}{2}\,\gamma_\ell,
\end{equation}
then an application of H\"older's inequality in \eqref{groviera} would lead us to
\[
\nabla u\in L^{2\,\gamma_\ell}(B_R) \qquad \Longrightarrow \qquad \nabla u \in L^{2\,\gamma_{\ell+1}}(B_r),
\] 
and we could enchain all the estimates. However, since the ratio $2^*/2$ tends to $1$ as the dimension $N$ goes to $\infty$, it is easy to see that \eqref{condizione0} cannot be true in general. More precisely, such a condition holds only up to dimension $N=4$.
\par
The idea is then to go back to \eqref{groviera} and use interpolation in Lebesgue spaces in order to construct a Moser's scheme ``{\it without holes}''. In a nutshell, we control the term
\[
\int_{B_R} |\nabla u|^{2\,\gamma_\ell}\,dx,
\]
with
\[
\int_{B_R} |\nabla u|^{2\,\gamma_{\ell-1}}\,dx\qquad \mbox{ and }\qquad \int_{B_R} |\nabla u|^{2^*\,\gamma_\ell}\,dx,
\]
and use an iteration over shrinking radii in order to absorb the last term, see {\bf Step 2} of the proof of Proposition \ref{prop:a_priori_estimate}. Once this is done, we end up with the updated self-improving information
\[
\nabla u\in L^{2\,\gamma_{\ell-1}}(B_R) \quad \Longrightarrow \quad \nabla u \in L^{2^*\gamma_\ell}(B_r).
\]
What we gain is that now $2^*\,\gamma_\ell> 2\,\gamma_\ell>2\,\gamma_{\ell-1}$, thus by using H\"older's inequality we obtain
\[
\nabla u\in L^{2\,\gamma_{\ell-1}}(B_R) \quad \Longrightarrow \quad \nabla u \in L^{2\,\gamma_\ell}(B_r).
\]
The information comes with a precise iterative estimate and a good control on the relevant constants. We can thus launch the Moser's iteration procedure and obtain the desired $L^\infty$ estimate, see {\bf Step 3} of the proof of Proposition \ref{prop:a_priori_estimate}.
\par
There is still a small detail that needs some care: the first exponent of the iteration is
\[
2\,\gamma_0=p+2,
\]
which means that on $\nabla u$ we obtain a $L^\infty-L^{p+2}$ local estimate. Finally, in order to obtain the desired $L^\infty-L^p$ estimate, one can simply use an interpolation argument (this is {\bf Step 4} of the proof of Proposition \ref{prop:a_priori_estimate}).

\subsection{Plan of the paper}

In Section \ref{sec:preliminaries}, we define the approximation scheme and settle all the needed machinery. 
We have dedicated Section \ref{sec:caccioppoli} to the new Caccioppoli inequalities which mix together the derivatives of the gradient with respect to $2$ orthogonal directions. In Section \ref{sec:leerp}, we exploit these Caccioppoli inequalities to establish integrability estimates on power functions of the gradient. In the subsequent section, we rely on these estimates to construct a Moser's iteration scheme which finally leads to the uniform a priori estimate of Proposition \ref{prop:a_priori_estimate}. 
\par
For ease of readability, both in Sections \ref{sec:leerp} and \ref{sec:5}, we first consider the case \(f=0\) and \(\delta=0\), in order to emphasize the main ideas and novelties of our approach. We explain subsequently in Subsections \ref{sec:leerp2} and \ref{subsec:ule2} respectively the technicalities to cover the general case \(f\in W^{1,h}_{\rm loc}(\Omega)\) and $\max\{\delta_i\, : \, i=1,\dots,N\} >0$. 
\par
Finally, in Appendix \ref{sec:nllot}, we generalize Theorem \ref{teo:lipschitz} to nonlinear lower order terms.

\begin{ack}
The paper has been partially written during a visit of P. B. \& L. B. to Napoli and of C. L. to Ferrara. Both visits have been funded by the Gruppo Nazionale per l'Analisi Matematica, la Pro\-ba\-bi\-li\-t\`a
e le loro Applicazioni (GNAMPA) through the project ``{\it Regolarit\`a per operatori degeneri con crescite generali\,}''. A further visit of P. B. to Ferrara in April 2017 has been the occasion to finalize the work.
Hosting institutions are gratefully acknowledged. 
\par
The last three authors are members of the Gruppo Nazionale per l'Analisi Matematica, la Pro\-ba\-bi\-li\-t\`a
e le loro Applicazioni (GNAMPA) of the Istituto Nazionale di Alta Matematica (INdAM).
\end{ack}

\section{Preliminaries}
\label{sec:preliminaries}

We will use the same approximation scheme as in \cite[Section 2]{BBJ}.
We introduce the notation
\[
g_i(t)=\frac{1}{p}\, (|t|-\delta_i)^p_+,\qquad t\in\mathbb{R},\ i=1,\dots,N,
\]
where $0\le \delta_1,\dots,\delta_N$ are given real numbers and we  also set
\begin{equation}
\label{delta}
\delta=1+\max\{\delta_i\, :\, i=1,\dots,N\}.
\end{equation}

We are interested in local minimizers of the following variational integral
\[
\mathfrak{F}_\delta(u;\Omega')=\sum_{i=1}^N \int_{\Omega'} g_i(u_{x_i})\, dx+\int_{\Omega'} f\, u\, dx,\qquad u\in W^{1,p}_{\rm loc}(\Omega),
\]
where $\Omega'\Subset\Omega$ and $f\in W^{1,h}_{\rm loc}(\Omega)$ for some \(h>N/2\). The latter implies that 
\[
f\in L^{h^*}_{\rm loc}(\Omega) \subset L^{N}_{\rm loc}(\Omega)  \subset L^{p'}_{\rm loc}(\Omega).
\] 
The last inclusion is a consequence of the fact that \(p\geq 2\) and \(N\geq 2\). The condition \(f\in L^{p'}_{\rm loc}\) is exactly the one required in \cite[Section 2]{BBJ}  to justify the approximation scheme that we now describe. 
\par
We set
\begin{equation}
\label{gepsilon}
g_{i,\varepsilon}(t)=g_i(t)+\frac{\varepsilon}{2}\, t^2=\frac{1}{p}\, (|t|-\delta_i)_+^{p}+\frac{\varepsilon}{2}\, t^2,\qquad t\in\mathbb{R}.
\end{equation}
\begin{oss}
For $p=2$ and $\delta_i>0$, we have $g_i\in C^{1,1}(\mathbb{R})\cap C^\infty(\mathbb{R}\setminus\{\delta_i,-\delta_i\})$, but $g_i$ is not $C^2$. In this case, like in \cite[Section 2]{BC} one would need to replace $g_i$ by a regularized version, in particular for the $C^2$ regularity result of Lemma \ref{lm:below} below. In order not to overburden the presentation, we prefer to avoid to explicitely write down this regularization and keep on using the same symbol $g_i$.
\end{oss}

From now on, {\it we fix $U$ a local minimizer of} $\mathfrak{F}_\delta$. We also fix a ball 
\[
B \Subset \Omega\quad \mbox{ such that }\quad 2\,B\Subset\Omega \mbox{ as well}.
\] 
Here $\lambda\,B$ denotes the ball having the same center as $B$, scaled by a factor $\lambda>0$.
\par
For every $0<\varepsilon\ll 1$ and every \(x\in \overline{B}\), we set $U_\varepsilon(x)=U\ast \varrho_\varepsilon(x)$, where $\varrho_\varepsilon$ is a smooth convolution kernel, supported in a ball of radius $\varepsilon$ centered at the origin. 
\par

Finally, we define
\[
\mathfrak{F}_{\delta,\varepsilon}(v;B)=\sum_{i=1}^N \int_B g_{i,\varepsilon}(v_{x_i})\, dx+\int_B f_{\varepsilon}\, v\, dx,
\]
where \(f_{\varepsilon}=f\ast\varrho_{\varepsilon}\). The following preliminary result is standard, see \cite[Lemma 2.5 and Lemma 2.8]{BBJ}.
\begin{lm}[Basic energy estimate]
\label{lm:below}
There exists $\varepsilon_0>0$ such that for every $0<\varepsilon\le \varepsilon_0<1$, the problem
\begin{equation}
\label{approximated}
\min\left\{\mathfrak{F}_\varepsilon(v;B)\, :\, v-U_\varepsilon\in W^{1,p}_0(B)\right\},
\end{equation}
admits a unique solution $u_\varepsilon$. 
Moreover, there exists a constant $C=C(N,p)>0$ such that the following uniform estimate holds 
\[
\int_B |\nabla u_\varepsilon|^p\, dx\le C\,\left[\int_{2\,B} |\nabla U|^p\,dx+|B|^\frac{p'}{N}\,\int_{2\,B} |f|^{p'}\,dx+(\varepsilon_0+(\delta-1)^p)|B|\,\right].
\]
Finally, $u_\varepsilon\in C^{2}(B)$.
\end{lm}

We also rely on the following stability result, which is slightly more precise than \cite[Lemma 2.9]{BBJ}.
\begin{lm}[Convergence to a minimizer]
\label{lm:convergence}
With the same notation as before, there exists a sequence $\{\varepsilon_k\}_{k\in\mathbb{N}}\subset(0,\varepsilon_0)$ converging to $0$, such that 
\[
\lim_{k\to \infty} \|u_{\varepsilon_k}-\widetilde u\|_{L^p(B)}=0,
\]
where $\widetilde u$ is a solution of
\[
\min\left\{\mathfrak{F}_\delta(v;B)\, :\, v-U\in W^{1,p}_0(B)\right\}.
\]
We also have
\begin{equation}
\label{propagation}
\Big|\widetilde u_{x_i} -U_{x_i}\Big| \leq 2\,\delta_i,\qquad \mbox{ for a.\,e. }x\in B,\quad i=1,\dots,N.
\end{equation}
In the case $\delta=1$, i.e. when $\delta_1=\dots=\delta_N=0$, then $\widetilde u=U$ and we have the stronger convergence
\begin{equation}
\label{troppoforte!}
\lim_{k\to \infty} \|u_{\varepsilon_k}-U\|_{W^{1,p}(B)}=0.
\end{equation}
\end{lm}
\begin{proof}
The first part is proven in \cite[Lemma 2.9]{BBJ}, while \eqref{propagation} is proven in \cite[Lemma 2.3]{BBJ}. For the case $\delta=1$, we observe that $\widetilde u=U$ follows from the strict convexity of the functional, together with the local minimality of $U$. In order to prove \eqref{troppoforte!}, we observe that
\[
\begin{split}
\left|\sum_{i=1}^N\frac{1}{p}\,\int_B \left|(u_{\varepsilon_k})_{x_i}\right|^p\,dx-\sum_{i=1}^N\frac{1}{p}\int_B \left|U_{x_i}\right|^p\,dx\right|&\le \left|\mathfrak{F}_{\delta,\varepsilon_k}(u_{\varepsilon_k};B)-\mathfrak{F}_\delta (U;B)\right|+\frac{\varepsilon_k}{2}\,\int_B |\nabla u_{\varepsilon_k}|^2\,dx\\
&+\left|\int_B f_{\varepsilon_k}\,u_{\varepsilon_k}\,dx-\int_B f\,U\,dx\right|.
\end{split}
\]
We now use that $\{u_{\varepsilon_k}\}_{k\in\mathbb{N}}$ strongly converges in $L^p(B)$, is bounded in $W^{1,p}(B)$ and that $\{f_{\varepsilon_k}\}_{k\in\mathbb{N}}$ strongly converges in $L^{p'}(B)$ to $f$. By further using that  (see the proof of \cite[Lemma 2.9]{BBJ})
\[
\lim_{k\to\infty}\left|\mathfrak{F}_{\delta,\varepsilon_k}(u_{\varepsilon_k};B)-\mathfrak{F}_\delta (U;B)\right|=0,
\]
we finally get
\begin{equation}
\label{norms}
\lim_{k\to\infty} \sum_{i=1}^N\int_B \left|(u_{\varepsilon_k})_{x_i}\right|^p\,dx=\sum_{i=1}^N\int_B \left|U_{x_i}\right|^p\,dx,\qquad i=1,\dots,N.
\end{equation}
Observe that by Clarkson's inequality for $p\ge 2$, we obtain
\[
\sum_{i=1}^N\left\|\frac{(u_{\varepsilon_k})_{x_i}+U_{x_i}}{2}\right\|^p_{L^p(B)}+\sum_{i=1}^N\left\|\frac{(u_{\varepsilon_k})_{x_i}-U_{x_i}}{2}\right\|^p_{L^p(B)}\le \frac{1}{2}\,\left(\sum_{i=1}^N \|(u_{\varepsilon_k})_{x_i}\|^p_{L^p(B)}+\sum_{i=1}^N \|U_{x_i}\|^p_{L^p(B)}\right).
\]
By using this and \eqref{norms}, we eventually get \eqref{troppoforte!}.
\end{proof}
\begin{oss}
Observe that the functional $\mathfrak{F}_\delta$ is not strictly convex when $\delta>1$. Thus property \eqref{propagation} is useful in order to transfer a Lipschitz estimate for the minimizer $\widetilde u$ selected in the limit, to the chosen one $U$. 
\end{oss}
Finally, we will repeatedly use the following classical result, see \cite[Lemma 6.1]{Gi} for a proof.
\begin{lm}
\label{lm:giusti}
Let $0<r<R$ and let $Z(t):[r,R]\to [0,\infty)$ be a bounded function. Assume that for $r\le t<s\le R$ we have
\[
Z(t)\le \frac{\mathcal{A}}{(s-t)^{\alpha_0}}+\frac{\mathcal{B}}{(s-t)^{\beta_0}}+\mathcal{C}+\vartheta\,Z(s),
\]
with $\mathcal{A},\mathcal{B},\mathcal{C}\ge 0$, $\alpha_0\ge \beta_0>0$ and $0\le \vartheta<1$. Then we have
\[
Z(r)\le \left(\frac{1}{(1-\lambda)^{\alpha_0}}\,\frac{\lambda^{\alpha_0}}{\lambda^{\alpha_0}-\vartheta}\right)\,\left[\frac{\mathcal{A}}{(R-r)^{\alpha_0}}+\frac{\mathcal{B}}{(R-r)^{\beta_0}}+\mathcal{C}\right],
\]
where $\lambda$ is any number such that
\[
\vartheta^\frac{1}{\alpha_0}<\lambda<1.
\]
\end{lm}

\section{Caccioppoli-type inequalities}

\label{sec:caccioppoli}

The solution $u_\varepsilon$ of the regularized problem \eqref{approximated} satisfies the Euler-Lagrange equation
\begin{equation}
\label{regolareg}
\sum_{i=1}^N \int g'_{i,\varepsilon}((u_\varepsilon)_{x_i})\, \varphi_{x_i}\, dx+\int f_\varepsilon\, \varphi\, dx=0,\qquad \varphi\in W^{1,p}_0(B).
\end{equation}
From now on, in order to simplify the notation, we will systematically forget the subscript $\varepsilon$ on $u_\varepsilon$  and \(f_\varepsilon\) and {\it simply write $u$ and \(f\) respectively}.
\par
We now insert a test function of the form $\varphi=\psi_{x_j}\in W^{1,p}_0(B)$ in \eqref{regolareg}, compactly supported in $B$. Then an integration by parts yields
\begin{equation}
\label{derivatag}
\sum_{i=1}^N \int g_{i,\varepsilon}''(u_{x_i})\, u_{x_i\,x_j}\, \psi_{x_i}\, dx+\int f_{x_j}\,\psi\,dx=0,
\end{equation}
for $j=1,\dots,N$. This is the equation solved by $u_{x_j}$. 
\vskip.2cm\noindent
We refer to \cite[Lemma 3.2]{BBJ} for a proof of the following Caccioppoli inequality:
\begin{lm}
Let $\Phi:\mathbb{R}\to\mathbb{R}^+$ be a $C^1$ convex function. 
Then there exists a constant $C=C(p)>0$ such that for every function $\eta\in C^\infty_0(B)$ and every $j=1,\dots,N$, we have
\begin{equation}
\label{mothergsob}
\begin{split}
\sum_{i=1}^N &\int g''_{i,\varepsilon}(u_{x_i})\,\left|\left(\Phi(u_{x_j})\right)_{x_i}\right|^2\, \eta^2\, dx\\
&\le C\,\sum_{i=1}^N \int g''_{i,\varepsilon}(u_{x_i})\,|\Phi(u_{x_j})|^2\, \eta_{x_i}^2\, dx+C\,\int |f_{x_j}|\, |\Phi'(u_{x_j})|\, |\Phi(u_{x_j})|\,\eta^2\, dx.\\
\end{split}
\end{equation}
\end{lm}
We need a more elaborate Caccioppoli-type inequality for the gradient, which is reminiscent of  \cite[Proposition 3.1]{BB}.

\begin{prop}[Weird Caccioppoli inequality]
\label{prop:weird}
Let \(\Phi, \Psi :[0,+\infty)\to [0,+\infty)\) be two non-decreasing continuous functions. We further assume that \(\Psi\) is convex and $C^1$. 
Let \(\eta\in C^{\infty}_0(B)\) and \(0\le \theta\le 2\), then for every \(k,j=1, \dots, N\),
\begin{equation}
\label{chiaraf}
\begin{split}
\sum_{i=1}^N \int g_{i,\varepsilon}''(u_{x_i})&\,u_{x_i x_j}^2\,\Phi(u_{x_j}^2)\,\Psi(u_{x_k}^2)\,\,\eta^2\,dx\\
&\le C\,\sum_{i=1}^N \int g_{i,\varepsilon}''(u_{x_i})\,u_{x_j}^2\,\Phi(u_{x_j}^2)\,\Psi(u_{x_k}^2)\,|\nabla \eta|^2\,dx\\
&+C\,\left(\sum_{i=1}^N \int g_{i,\varepsilon}''(u_{x_i})\,u_{x_i x_j}^2\,u_{x_j}^2\,\Phi(u_{x_j}^2)^2\,\Psi'(u_{x_k}^2)^\theta \,\eta^2\,dx\right)^\frac{1}{2}\\
&\times\left[\left(\sum_{i=1}^N\int g_{i,\varepsilon}''(u_{x_i})\,|u_{x_k}|^{2\theta}\,\Psi(u_{x_k}^2)^{2-\theta}\,|\nabla \eta|^2\,dx \right)^\frac{1}{2}+  \mathcal{E}_1(f)^\frac{1}{2}\right] + C\,\mathcal{E}_2(f)
\end{split}
\end{equation}
where 
\[
\mathcal{E}_1(f):=\int |f_{x_k}|\, |u_{x_k}|^{\theta+1}\,\Big|\Psi(u_{x_k}^2)\,\Psi'(u_{x_k}^2)\Big|^{1-\frac{\theta}{2}}\,\eta^2\,dx,
\]
\[
\mathcal{E}_2(f):=\int |f_{x_j}|\,|u_{x_j}|\,\Phi(u_{x_j}^2)\,\Psi(u_{x_k}^2)\,\eta^2\,dx.
\]
\end{prop}
\begin{proof}
By a standard approximation argument, one can assume that \(\Phi\)  is \(C^1\) as well.
We  take in \eqref{derivatag}
\[
\varphi=u_{x_j}\,\Phi(u_{x_j}^2)\,\Psi(u_{x_k}^2)\,\eta^2.
\] 
This gives
\begin{equation}
\label{carnevalif}
\begin{split}
\sum_{i=1}^N \int &g_{i,\varepsilon}''(u_{x_i})\,u_{x_i x_j}^2\,\Big(\Phi(u_{x_j}^2)+ 2u_{x_j}^2\,\Phi'(u_{x_j}^2)\Big)\,\Psi(u_{x_k}^2)\,\,\eta^2\,dx\\
&=-2\,\sum_{i=1}^N \int g_{i,\varepsilon}''(u_{x_i})\,u_{x_i x_j}\,u_{x_j}\,\Phi(u_{x_j}^2)\,\Psi(u_{x_k}^2)\,\,\eta\,\eta_{x_i}\,dx\\
&-2\,\sum_{i=1}^N \int g_{i,\varepsilon}''(u_{x_i})\,u_{x_i x_j}\,u_{x_j}\,u_{x_i x_k}\,u_{x_k}\,\Psi'(u_{x_k}^2)\,\Phi(u_{x_j}^2)\,\,\eta^2\,dx\\
&-\int f_{x_j}\,u_{x_j}\,\Phi(u_{x_j}^2)\,\Psi(u_{x_k}^2)\,\eta^2\,dx=:\mathcal{A}_1 +\mathcal{A}_2+\mathcal{A}_3.
\end{split}
\end{equation}
We now proceed to estimating the three terms $\mathcal{A}_\ell$. We have
\[
\begin{split}
\mathcal{A}_1&\le \frac{1}{2}\,\sum_{i=1}^N \int g_{i,\varepsilon}''(u_{x_i})\,u_{x_i x_j}^2\,\Phi(u_{x_j}^2)\,\Psi(u_{x_k}^2)\,\eta^2\,dx\\
&+2\, \sum_{i=1}^N \int g_{i,\varepsilon}''(u_{x_i})\,u_{x_j}^2\,\Phi(u_{x_j}^2)\,\Psi(u_{x_k}^2)\,\eta_{x_i}^2\,dx
\end{split}
\]
and the integral containing the Hessian of $u$ can be absorbed in the left-hand side of \eqref{carnevalif}. Using also that \(2\,u_{x_j}^2\,\Phi'(u_{x_j}^2) \geq 0\), this yields
\begin{equation}
\label{carnevali2f}
\begin{split}
\frac{1}{2}\,\sum_{i=1}^N \int_{\Omega} &g_{i,\varepsilon}''(u_{x_i})\,u_{x_i x_j}^2\,\Phi(u_{x_j}^2)\,\Psi(u_{x_k}^2)\,\,\eta^2\,dx\\
&\le 2\,\sum_{i=1}^N \int g_{i,\varepsilon}''(u_{x_i})\,u_{x_j}^2\,\Phi(u_{x_j}^2)\,\Psi(u_{x_k}^2)\,\eta_{x_i}^2\,dx+\mathcal{A}_2+\mathcal{A}_3.
\end{split}
\end{equation}
We now estimate \(\mathcal{A}_2\), which is the most delicate term: writing \(\Psi'(u_{x_k}^2)=\Psi'(u_{x_k}^2)^{\frac{\theta}{2}}\,\Psi'(u_{x_k}^2)^{1-\frac{\theta}{2}}\) and using Cauchy-Schwarz inequality, we get
\begin{equation}
\label{eqferialif}
\begin{split}
\mathcal{A}_2&\le 2\,\left(\sum_{i=1}^N \int g_{i,\varepsilon}''(u_{x_i})\,u_{x_i x_j}^2\,u_{x_j}^2\,\Phi(u_{x_j}^2)^2\,\Psi'(u_{x_k}^2)^{\theta}\,\eta^2\,dx\right)^\frac{1}{2}\\
&\times\left(\sum_{i=1}^N\int g_{i,\varepsilon}''(u_{x_i})\,u_{x_i x_k}^2\,u_{x_k}^2\,\Psi'(u_{x_k}^2)^{2-\theta}\,\eta^2\,dx\right)^\frac{1}{2}.
\end{split}
\end{equation}
We observe that
\[
\sum_{i=1}^N\int g_{i,\varepsilon}''(u_{x_i})\,u_{x_i x_k}^2\,u_{x_k}^2\,\Psi'(u_{x_k}^2)^{2-\theta}\,\eta^2\,dx=\frac{1}{4}\,\sum_{i=1}^N\int g_{i,\varepsilon}''(u_{x_i})\,\left|\left(G(u_{x_k})\right)_{x_i}\right|^2\,\eta^2\,dx,
\]
where \(G\) is the convex nonnegative \(C^1\) function defined by
\[
G(t)=\int_0^{t^2} \Psi'(\tau)^{1-\frac{\theta}{2}}\,d\tau.
\]
Thus by Caccioppoli inequality \eqref{mothergsob} with $x_k$ in place of $x_j$ and
\[
\Phi(t)=G(t),\qquad t\in\mathbb{R},
\]
we get
\[
\begin{split}
\sum_{i=1}^N\int g_{i,\varepsilon}''(u_{x_i})\,u_{x_i x_k}^2\,u_{x_k}^2\,\Psi'(u_{x_k}^2)^{2-\theta}\,\eta^2
&\le C\sum_{i=1}^N\int g_{i,\varepsilon}''(u_{x_i})\,G(u_{x_k})^2\,\eta_{x_i}^2\,dx\\
& + C\,\int |f_{x_k}|\,\Big|G(u_{x_k})\, G'(u_{x_k})\Big|\,\eta^2\,dx.
\end{split}
\]
By Jensen's inequality 
\[
0\le G(u_{x_k})\le |u_{x_k}|^{\theta}\left(\int_{0}^{u_{x_k}^2} \Psi'(\tau)\,d\tau \right)^{1-\frac{\theta}{2}}\le |u_{x_k}|^{\theta}\,\Psi(u_{x_k}^2)^{1-\frac{\theta}{2}}.
\]
Together with the fact that \(G'(u_{x_k})=2\,u_{x_k}\Psi'(u_{x_k}^2)^{1-\frac{\theta}{2}}\), this implies
\[
\begin{split}
\sum_{i=1}^N\int g_{i,\varepsilon}''(u_{x_i})&\,u_{x_i x_k}^2\,u_{x_k}^2\,\Psi'(u_{x_k}^2)^{2-\theta}\,\eta^2
\le C\,\sum_{i=1}^N\int g_{i,\varepsilon}''(u_{x_i})\,|u_{x_k}|^{2\theta}\,\Psi(u_{x_k}^2)^{2-\theta}\,\eta_{x_i}^2\,dx\\
&+ C\,\int |f_{x_k}|\,|u_{x_k}|^{\theta+1}\,\Big|\Psi(u_{x_k}^2)\,\Psi'(u_{x_k}^2)\Big|^{1-\frac{\theta}{2}}\,\eta^2\,dx,
\end{split}
\]
which in turn yields by \eqref{carnevali2f} and \eqref{eqferialif},
\[
\begin{split}
\frac{1}{2}\,\sum_{i=1}^N \int_{\Omega} &g_{i,\varepsilon}''(u_{x_i})\,u_{x_i x_j}^2\,\Phi(u_{x_j}^2)\,\Psi(u_{x_k}^2)\,\,\eta^2\,dx\\
&\le 2\,\sum_{i=1}^N \int g_{i,\varepsilon}''(u_{x_i})\,u_{x_j}^2\,\Phi(u_{x_j}^2)\,\Psi(u_{x_k}^2)\,\eta_{x_i}^2\,dx\\
&+ C\,\left(\sum_{i=1}^N \int g_{i,\varepsilon}''(u_{x_i})\,u_{x_i x_j}^2\,u_{x_j}^2\,\Phi(u_{x_j}^2)^2\,\Psi'(u_{x_k}^2)^{\theta}\,\eta^2\,dx\right)^\frac{1}{2}\\
&\times \left[\left(\sum_{i=1}^N\int g_{i,\varepsilon}''(u_{x_i})\,|u_{x_k}|^{2\theta}\,\Psi(u_{x_k}^2)^{2-\theta}\,\eta_{x_i}^2\,dx \right)^\frac{1}{2}\right.\\
&\qquad \left.+ \left(\int |f_{x_k}|\,|u_{x_k}|^{\theta+1}\,\Big|\Psi(u_{x_k}^2)\,\Psi'(u_{x_k}^2)\Big|^{1-\frac{\theta}{2}}\,\eta^2\,dx\right)^\frac{1}{2}\right]+\mathcal{A}_3.
\end{split}
\]
Here, we have also used the inequality \((A+B)^{1/2} \leq A^{1/2} + B^{1/2}.\)

Finally, 
\[
\mathcal{A}_3 \leq  C\,\int_{\Omega} |f_{x_j}|\,|u_{x_j}|\,\Phi(u_{x_j}^2)\,\Psi(u_{x_k}^2)\,\eta^2\,dx.
\]
This completes the proof.
\end{proof}

\section{Local energy estimates for the regularized problem}
\label{sec:leerp}
In order to emphasize the main ideas of the proof, we have divided this section in two parts. In the first one, we  explain how \eqref{chiaraf} leads to higher integrability estimates for the gradient when \(f=0\) and  \(\delta=1\). This allows to ignore a certain amount of technicalities. In the second part, we then detail the  modifications of the proof to obtain the corresponding estimates in the general case.

\subsection{The homogeneous case}
In this subsection, we assume that \(f=0\) and \(\delta=1\). Then the two terms \(\mathcal{E}_1(f)\) and \(\mathcal{E}_2(f)\) in \eqref{chiaraf} vanish. Also observe that in this case from \eqref{gepsilon} we have
\[
g_{i,\varepsilon}''(t)=(p-1)\,|t|^{p-2}+\varepsilon.
\]
Let us single out a particular case of Proposition \ref{prop:weird} by taking
\begin{equation}
\label{sceltechiare}
\Phi(t)=t^{s-1}\qquad \mbox{ and }\qquad \Psi(t)=t^m,\qquad \mbox{ for }t\ge 0,
\end{equation}
with $1\le s \le m$.
\begin{prop}[Staircase to the full Caccioppoli]\label{stair}
Let $p\ge 2$ and let \(\eta\in C^{\infty}_0(B)\), then for every \(k,j=1, \dots, N\) and $1\le s\le m$
\begin{equation}
\label{powerchiara}
\begin{split}
\sum_{i=1}^N \int_{\Omega} g_{i,\varepsilon}''(u_{x_i})\,u_{x_i x_j}^2\,|u_{x_j}|^{2\,s-2}\,|u_{x_k}|^{2\,m}\,\,\eta^2\,dx&\le C\,\sum_{i=1}^N \int g_{i,\varepsilon}''(u_{x_i})\,|u_{x_j}|^{2\,s+2\,m}\,|\nabla \eta|^2\,dx\\
& + C\,(m+1)\,\sum_{i=1}^N \int g_{i,\varepsilon}''(u_{x_i})\,|u_{x_k}|^{2\,s+2\,m}\,|\nabla \eta|^2\,dx\\
& + \sum_{i=1}^N \int g_{i,\varepsilon}''(u_{x_i})\,u_{x_i x_j}^2\,|u_{x_j}|^{4\,s-2}\,|u_{x_k}|^{2\,m-2\,s}\,\eta^2\,dx.
\end{split}
\end{equation}
\end{prop}
\begin{proof}
We use \eqref{chiaraf} with the choices \eqref{sceltechiare} above and 
\[
\theta = 
\begin{cases}
\dfrac{m-s}{m-1} \in [0,1] & \textrm{ if } m>1,\\
&\\
1 & \textrm{ if } m=1.
\end{cases}
\] 
This gives
\[
\begin{split}
\sum_{i=1}^N \int_{\Omega} &g_{i,\varepsilon}''(u_{x_i})\,u_{x_i x_j}^2\,|u_{x_j}|^{2\,s-2}\,|u_{x_k}|^{2\,m}\,\,\eta^2\,dx\\
&\le C\,\sum_{i=1}^N \int g_{i,\varepsilon}''(u_{x_i})\,|u_{x_j}|^{2\,s}\,|u_{x_k}|^{2\,m}\,|\nabla \eta|^2\,dx\\
&+C\,\left( m^{\theta}\,\sum_{i=1}^N \int g_{i,\varepsilon}''(u_{x_i})\,u_{x_i x_j}^2\,|u_{x_j}|^{4\,s-2}\,|u_{x_k}|^{2\,m-2\,s}\,\eta^2\,dx\right)^\frac{1}{2}\\
&\times\left(\,\sum_{i=1}^N\int_\Omega g_{i,\varepsilon}''(u_{x_i})\,|u_{x_k}|^{2\,m+2\,s}\,|\nabla \eta|^2\,dx\right)^\frac{1}{2}.
\end{split}
\]
We use Young's inequality in the form \(C\,\sqrt{a\,b}\leq C^2\, b/4 +a\) for the product in the right-hand side to get
\[
\begin{split}
\sum_{i=1}^N \int_{\Omega} &g_{i,\varepsilon}''(u_{x_i})\,u_{x_i x_j}^2\,|u_{x_j}|^{2\,s-2}\,|u_{x_k}|^{2\,m}\,\,\eta^2\,dx\\
&\le C\,\sum_{i=1}^N \int g_{i,\varepsilon}''(u_{x_i})\,|u_{x_j}|^{2\,s}\,|u_{x_k}|^{2\,m}\,|\nabla \eta|^2\,dx\\
&+C\,m^{\theta}\,\sum_{i=1}^N\int_\Omega g_{i,\varepsilon}''(u_{x_i})\,|u_{x_k}|^{2\,m+2\,s}\,|\nabla \eta|^2\,dx\\
&+\sum_{i=1}^N \int g_{i,\varepsilon}''(u_{x_i})\,u_{x_i x_j}^2\,|u_{x_j}|^{4\,s-2}\,|u_{x_k}|^{2\,m-2\,s}\,\eta^2\,dx.
\end{split}
\]
In the  first term of the right-hand side, we use Young's inequality with the exponents
\[
\frac{2\,m+2\,s}{2\,s} \quad , \quad \frac{2\,m+2\,s}{2\,m}.
\]
We also observe for the second term that \(m^\theta\leq m\).
This gives the desired estimate.
\end{proof}
\begin{prop}[Caccioppoli for power functions of the gradient]
\label{prop-russian}
We fix an exponent 
\[
q=2^{\ell_0}-1,\qquad \mbox{ for a given } \ell_0\in\mathbb{N}\setminus\{0\}.
\] 
Let \(\eta\in C^{\infty}_0(B)\), then for every \(k=1, \dots, N\) we have
\begin{equation}
\label{russiancircles}
\begin{split}
\int \left|\nabla \left(|u_{x_k}|^{q+\frac{p-2}{2}}\,u_{x_k}\right)\right|^2\,\eta^2\,dx
&\le C\,q^5\,\sum_{i,j=1}^N\int g_{i,\varepsilon}''(u_{x_i})\,|u_{x_j}|^{2\,q+2}\,|\nabla \eta|^2\,dx\\
& + C\,q^5\,\sum_{i=1}^N\int g_{i,\varepsilon}''(u_{x_i})\,|u_{x_k}|^{2\,q+2}\,|\nabla \eta|^2\,dx,
\end{split}
\end{equation}
for some $C=C(N,p)>0$.
\end{prop}
\begin{proof}
We define the two finite families of indices $\{s_\ell\}$ and $\{m_\ell\}$ such that
\[
s_\ell=2^\ell,\qquad m_{\ell}=q+1-2^{\ell},\qquad \ell\in\{0,\dots,\ell_0\}.
\]
Observe that 
\[
1\le s_\ell\le m_\ell,\qquad \ell\in\{0,\dots,\ell_0-1\},
\]
\[
s_\ell+m_\ell=q+1,\qquad \ell\in\{0,\dots,\ell_0\},
\]
\[
4\,s_\ell-2=2\,s_{\ell+1}-2,\qquad 2\,m_\ell-2\,s_\ell=2\,m_{\ell+1},
\]
and
\[
s_0=1,\qquad m_0=q,\qquad s_{\ell_0}=2^{\ell_0},\qquad m_{\ell_0}=0. 
\]
In terms of these families, inequality \eqref{powerchiara} implies  for every $\ell\in\{0,\dots,\ell_0-1\}$
\[
\begin{split}
\sum_{i=1}^N \int &g_{i,\varepsilon}''(u_{x_i})\,u_{x_i x_j}^2\,|u_{x_j}|^{2\,s_\ell-2}\,|u_{x_k}|^{2\,m_\ell}\,\,\eta^2\,dx\\
&\le C\,\sum_{i=1}^N \int g_{i,\varepsilon}''(u_{x_i})\,|u_{x_j}|^{2\,q+2}\,|\nabla \eta|^2\,dx\\
&+ C\,(m_\ell+1)\,\sum_{i=1}^N \int g_{i,\varepsilon}''(u_{x_i})\,| u_{x_k}|^{2\,q+2}\,|\nabla \eta|^2\,dx\\
&+\sum_{i=1}^N \int g_{i,\varepsilon}''(u_{x_i})\,u_{x_i x_j}^2\,|u_{x_j}|^{2\,s_{\ell+1}-2}\,|u_{x_k}|^{2\,m_{\ell+1}}\,\eta^2\,dx,
\end{split}
\]
for some $C>0$ universal.
By starting from $\ell=0$ and iterating the previous estimate up to $\ell=\ell_0-1$, we then get
\[
\begin{split}
\sum_{i=1}^N \int g_{i,\varepsilon}''(u_{x_i})\,u_{x_i x_j}^2\,|u_{x_k}|^{2\,q}\,\eta^2\,dx&\le C\,q^2\,\sum_{i=1}^N \int g_{i,\varepsilon}''(u_{x_i})\,|u_{x_j}|^{2\,q+2}\,|\nabla \eta|^2\,dx\\
&+ C\,q^2\,\sum_{i=1}^N\int g_{i,\varepsilon}''(u_{x_i})\,|u_{x_k}|^{2\,q+2}\,|\nabla \eta|^2\,dx\\
&+\sum_{i=1}^N \int g_{i,\varepsilon}''(u_{x_i})\,u_{x_i x_j}^2\,|u_{x_j}|^{2\,q}\,\eta^2\,dx,
\end{split}
\]
for a universal constant $C>0$.
For the last term, we apply the Caccioppoli inequality \eqref{mothergsob}  with
\[
\Phi(t)=\frac{|t|^{q+1}}{q+1},\qquad t\in\mathbb{R},
\]
thus we get
\[
\begin{split}
\sum_{i=1}^N \int g_{i,\varepsilon}''(u_{x_i})\,u_{x_i x_j}^2\,|u_{x_k}|^{2\,q}\,\eta^2\,dx
&\le C\,q^2\,\sum_{i=1}^N\int g_{i,\varepsilon}''(u_{x_i})\,|u_{x_j}|^{2\,q+2}\,|\nabla \eta|^2\,dx\\
& + C\,q^2\,\sum_{i=1}^N\int g_{i,\varepsilon}''(u_{x_i})\,|u_{x_k}|^{2\,q+2}\,|\nabla \eta|^2\,dx\\
&+\frac{C}{(q+1)^2}\,\sum_{i=1}^N\int g_{i,\varepsilon}''(u_{x_i})\,|u_{x_j}|^{2\,q+2}\,|\nabla \eta|^2\,dx;
\end{split}
\]
that is,
\begin{equation}
\label{eq_yoyo}
\begin{split}
\sum_{i=1}^N \int g_{i,\varepsilon}''(u_{x_i})\,u_{x_i x_j}^2\,|u_{x_k}|^{2\,q}\,\eta^2\,dx
&\le C\,q^2\,\sum_{i=1}^N\int g_{i,\varepsilon}''(u_{x_i})\,|u_{x_j}|^{2\,q+2}\,|\nabla \eta|^2\,dx\\
& + C\,q^2\,\sum_{i=1}^N\int g_{i,\varepsilon}''(u_{x_i})\,|u_{x_k}|^{2\,q+2}\,|\nabla \eta|^2\,dx,
\end{split}
\end{equation}
possibly for a different universal constant $C>0$.
\par
We now observe that \(g_{i,\varepsilon}''(u_{x_i})=\Big((p-1)\,|u_{x_i}|^{p-2}+\varepsilon\Big)\) and thus
\[
\begin{split}
\sum_{i=1}^N \int g_{i,\varepsilon}''(u_{x_i})\,u_{x_i x_j}^2\,|u_{x_k}|^{2\,q}\,\eta^2\,dx&\ge \int |u_{x_k}|^{p-2}\,u_{x_k x_j}^2\,|u_{x_k}|^{2\,q}\,\eta^2\,dx\\
&=\left(\frac{2}{2\,q+p}\right)^2\,\int \left|\left(|u_{x_k}|^{q+\frac{p-2}{2}}\,u_{x_k}\right)_{x_j}\right|^2\,\eta^2\,dx.
\end{split}
\]
When we sum over $j=1,\dots,N$, we get
\[
\sum_{i,j=1}^N \int {g_{i,\varepsilon}''(u_{x_i})}\,u_{x_i x_j}^2\,|u_{x_k}|^{2\,q}\,\eta^2\,dx\ge \left(\frac{2}{2\,q+p}\right)^2\,\int \left|\nabla \left(|u_{x_k}|^{q+\frac{p-2}{2}}\,u_{x_k}\right)\right|^2\,\eta^2\,dx.
\]
This proves the desired inequality.
\end{proof}

\subsection{The non-homogeneous case}
\label{sec:leerp2}
In the general case where \(f\not=0\) and/or \(\delta>1\), we can prove the following  analogue of  \eqref{powerchiara},  in a similar way: 

\begin{equation}
\label{powerchiaraf}
\begin{split}
\sum_{i=1}^N \int_{\Omega} g_{i,\varepsilon}''(u_{x_i})&\,u_{x_i x_j}^2\,|u_{x_j}|^{2\,s-2}\,|u_{x_k}|^{2\,m}\,\,\eta^2\,dx\\
 \le &  \sum_{i=1}^N \int g_{i,\varepsilon}''(u_{x_i})\,u_{x_i x_j}^2\,|u_{x_j}|^{4\,s-2}\,|u_{x_k}|^{2\,m-2\,s}\,\eta^2\,dx\\
&+ C\,{(m+1)}\,\sum_{i=1}^N \int g_{i,\varepsilon}''(u_{x_i})\,\left(|u_{x_j}|^{2\,s+2\,m}+|u_{x_k}|^{2\,s+2\,m}\right)\,|\nabla \eta|^2\,dx\\
&+C\, m^2\, \int |\nabla f|\,\left(|u_{x_k}|^{2\,s+2\,m-1}+|u_{x_j}|^{2\,s+2\,m-1}\right)\,\eta^2\,dx.
\end{split}
\end{equation}
We then deduce the following analogue of Proposition \ref{prop-russian}:

\begin{prop}
We fix an exponent 
\[
q=2^{\ell_0}-1,\qquad \mbox{ for a given } \ell_0\in\mathbb{N}\setminus\{0\}.
\] 
Let \(\eta\in C^{\infty}_0(\Omega)\), then for every \(k=1, \dots, N\) we have
\begin{equation}
\label{russiancirclesf}
\begin{split}
\int \left|\nabla \left((|u_{x_k}|-\delta_k)_{+}^{\frac{p}{2}}\,|u_{x_k}|^q\right)\right|^2\,\eta^2\,dx
&\le C\,q^5\,\sum_{i=1}^N\int g_{i,\varepsilon}''(u_{x_i})\,\left(|u_{x_k}|^{2\,q+2}+\sum_{j=1}^N|u_{x_j}|^{2\,q+2}\right)\,|\nabla \eta|^2\,dx\\
&+C\,q^5\, \int |\nabla f|\, \left(|u_{x_k}|^{2\,q+1} + \sum_{j=1}^{N}|u_{x_j}|^{2\,q+1}\right)\,\eta^2\,dx,
\end{split}
\end{equation}
for some $C=C(N,p)>0$.
\end{prop}
\begin{proof}
Using the same notation and the same strategy as in the proof of \eqref{russiancircles}, except that we start from \eqref{powerchiaraf} instead of \eqref{powerchiara}, we get the following analogue of \eqref{eq_yoyo}:
\[
\begin{split}
\sum_{i=1}^N \int& g_{i,\varepsilon}''(u_{x_i})\,u_{x_i x_j}^2\,|u_{x_k}|^{2\,q}\,\eta^2\,dx\\
&\le C\,q^2\,\sum_{i=1}^N\int g_{i,\varepsilon}''(u_{x_i})\,(|u_{x_j}|^{2\,q+2}+|u_{x_k}|^{2\,q+2})\,|\nabla \eta|^2\,dx\\
&+C\,q^3\, \int |\nabla f|\, (|u_{x_k}|^{2\,q+1}+ |u_{x_j}|^{2\,q+1})\,\eta^2\,dx.
\end{split}
\]
We now observe that 
\[
\sum_{i=1}^N \int g_{i,\varepsilon}''(u_{x_i})\,u_{x_i x_j}^2\,|u_{x_k}|^{2\,q}\,\eta^2\,dx\ge (p-1)\,\int (|u_{x_k}|-\delta_k)_{+}^{p-2}\,u_{x_k x_j}^2\,|u_{x_k}|^{2\,q}\,\eta^2\,dx.
\]
Noting that
\[
(|u_{x_k}|-\delta_k)_+^{p}\leq (|u_{x_k}|-\delta_k)_+^{p-2}|u_{x_k}|^2,
\]
we have 
\[
\begin{split}
\left|\left((|u_{x_k}|-\delta_k)_{+}^{\frac{p}{2}}\,|u_{x_k}|^q\right)_{x_j}\right|^2
&\leq 2\,\left|\left((|u_{x_k}|-\delta_k)_{+}^{\frac{p}{2}}\right)_{x_j}\right|^2\,|u_{x_k}|^{2\,q} +  2\,(|u_{x_k}|-\delta_k)_{+}^{p}\,\left|\left(|u_{x_k}|^q\right)_{x_j}\right|^2\\
&\leq C\,q^2\,(|u_{x_k}|-\delta_k)_{+}^{p-2}\,|u_{x_k}|^{2\,q}\,u_{x_k x_j}^2.
\end{split}
\]
We deduce therefrom
\[
\sum_{i=1}^N \int g_{i,\varepsilon}''(u_{x_i})\,u_{x_i x_j}^2\,|u_{x_k}|^{2\,q}\,\eta^2\,dx
\geq \frac{C}{q^2}\,\int \left|\left((|u_{x_k}|-\delta_k)_{+}^{\frac{p}{2}}\,|u_{x_k}|^q\right)_{x_j}\right|^2\,\eta^2\,dx,
\]
thus when we sum over $j=1,\dots,N,$ we get
\[
\sum_{i,j=1}^N \int g_{i,\varepsilon}''(u_{x_i}) \,u_{x_i x_j}^2\,|u_{x_k}|^{2\,q}\,\eta^2\,dx\ge \frac{C}{q^2}\,\int \left|\nabla \left((|u_{x_k}|-\delta_k)_{+}^{\frac{p}{2}}\,|u_{x_k}|^q\right)\right|^2\,\eta^2\,dx.
\]
This proves the desired inequality \eqref{russiancirclesf}.
\end{proof}

\section{Proof of Theorem \ref{teo:lipschitz}}
\label{sec:5}

\begin{proof}
The core of the proof of Theorem \ref{teo:lipschitz} is the uniform Lipschitz estimate of Proposition \ref{prop:a_priori_estimate} below. Its proof, which is postponed for ease of readability, uses the integrability estimates of Section \ref{sec:leerp}. Once we have this uniform estimate, we can reproduce the proof of \cite[Theorem A]{BBJ} and prove that $\nabla U\in L^\infty(\Omega')$, for every $\Omega'\Subset\Omega$. 
\par
We now detail how to obtain the scaling invariant local estimate \eqref{stimayeah} in the case $\delta_1=\dots=\delta_N=0$. We take $0<r_0<R_0\le 1$ and a ball $B_{2R_0}\Subset \Omega$. We then consider the sequence of miminizers $\{u_{\varepsilon_k}\}_{k\in\mathbb{N}}$ of \eqref{approximated} obtained in Lemma \ref{lm:convergence}, with $B$ a ball slightly larger than $B_{R_0}$ so that $2\,B\Subset \Omega$.
By using the uniform Lipschitz estimate \eqref{lipschitzf} below, taking the limit as $k$ goes to $\infty$  and using the strong convergence of Lemma \ref{lm:convergence}, we obtain
\[
\|\nabla U\|_{L^\infty(B_{r_0})}\le \frac{C}{(R_0-r_0)^{\sigma_2}}\,\left(1+\|\nabla f\|^{\sigma_2}_{L^{h}(B_{R_0})}\right)\,\left(\|\nabla U\|^{\sigma_1}_{L^{p}(B_{R_0})}+1\right).
\]
Without loss of generality, we can assume that \(\|\nabla U\|_{L^{p}(B_{R_0})}>0\). Hence, by Young's inequality,
\begin{equation}
\label{intermedia}
\|\nabla U\|_{L^\infty(B_{r_0})}\le \frac{C}{(R_0-r_0)^{\sigma_2}}\,\left(1+\|\nabla f\|_{L^{h}(B_{R_0})}^{2\,\sigma_2}+\|\nabla U\|_{L^{p}(B_{R_0})}^{2\,\sigma_1}\right),
\end{equation}
possibly for a different $C=C(N,p,h)>0$.
We now observe that for every  \(\lambda>0\), \(\lambda\,U\) is still a solution of the orthotropic \(p-\)Laplace equation, with the right hand side \(f\) replaced by \(\lambda^{p-1}\,f\). We can use \eqref{intermedia} for $\lambda\, U$ and get
\[
\lambda\,\|\nabla U\|_{L^\infty(B_{r_0})}\le \frac{C}{(R_0-r_0)^{\sigma_2}}\,\left(1+\lambda^{2\,\sigma_2\,(p-1)}\,\|\nabla f\|_{L^{h}(B_{R_0})}^{2\,\sigma_2}+\lambda^{2\,\sigma_1}\,\|\nabla U\|_{L^{p}(B_{R_0})}^{2\,\sigma_1}\right).
\]
Dividing by \(\lambda\), we obtain
\[
\|\nabla U\|_{L^\infty(B_{r_0})}\le \frac{C}{(R_0-r_0)^{\sigma_2}}\,\left(\frac{1}{\lambda}+\lambda^{2\,\sigma_2\,(p-1)-1}\,\|\nabla f\|^{2\sigma_2}_{L^{h}(B_{R_0})}+\lambda^{2\,\sigma_1-1}\,\|\nabla U\|_{L^{p}(B_{R_0})}^{2\,\sigma_1}\right).
\]
We take
\[
\lambda:=\frac{1}{\|\nabla U\|_{L^{p}(B_{R_0})} + \|\nabla f\|_{L^{h}(B_{R_0})}^{\frac{1}{p-1}}},
\]
and observe that if \(\|\nabla f\|_{L^{h}(B_{R_0})}>0\), then
\[
\lambda^{2\,\sigma_2\,(p-1)-1}\,\|\nabla f\|^{2\sigma_2}_{L^{h}(B_{R_0})}\leq \frac{1}{ \left(\|\nabla f\|_{L^{h}(B_{R_0})}^{\frac{1}{p-1}}\right)^{2\,\sigma_2\,(p-1)-1}}\,\|\nabla f\|^{2\sigma_2}_{L^{h}(B_{R_0})}=\|\nabla f\|_{L^{h}(B_{R_0})}^{\frac{1}{p-1}}
\]
while the  inequality is obvious when \(\|\nabla f\|_{L^{h}(B_{R_0})}=0\). Similarly, 
\[
\lambda^{2\,\sigma_1-1}\,\|\nabla U\|_{L^{p}(B_{R_0})}^{2\,\sigma_1}\leq \frac{1}{\|\nabla U\|_{L^{p}(B_{R_0})}^{2\,\sigma_1-1}}\,\|\nabla U\|_{L^{p}(B_{R_0})}^{2\,\sigma_1}=\|\nabla U\|_{L^{p}(B_{R_0})}.
\]
It thus follows that
\begin{equation}\label{altezzec}
\|\nabla U\|_{L^\infty(B_{r_0})}\le \frac{C}{(R_0-r_0)^{\sigma_2}}\,\left(\|\nabla f\|_{L^{h}(B_{R_0})}^{\frac{1}{p-1}} + \|\nabla U\|_{L^{p}(B_{R_0})}\right).
\end{equation}
We now make this estimate dimensionally correct. Given $R_0>0$, we consider a ball $B_{2R_0}\Subset\Omega$. Then the rescaled function
\[
U_{R_0}(x)=U(R_0\,x),\qquad \mbox{ for }x\in R_0^{-1}\,\Omega,
\] 
is a solution of the orthotropic $p-$Laplace equation, with right-hand side \(f_{R_0}(x):=R_{0}^p\,f(R_0\,x)\).
We can use for it the estimate \eqref{altezzec} with radii $1$ and $1/2$. By scaling back, we thus obtain
\[
R_0\,\|\nabla U\|_{L^\infty(B_{R_0/2})}\le C\,\left( R_0^{-\frac{N}{p}+1}\,\|\nabla U\|_{L^{p}(B_{R_0})}+ R_{0}^{\frac{h\,(p+1)-N}{h\,(p-1)}}\,\|\nabla f\|_{L^{h}(B_{R_0})}^{\frac{1}{p-1}} \right),
\]
for some constant $C=C(N,p,h)>1$. Dividing by \(R_0\), we get
\[
\|\nabla U\|_{L^\infty(B_{R_0/2})}\le C\,\left(\fint_{B_{R_0}}|\nabla U|^{p}\,dx\right)^\frac{1}{p}+ C\,R_{0}^{\frac{2}{p-1}-\frac{N}{h\,(p-1)}}\left( \int_{B_{R_0}}|\nabla f|^{h}\,dx\right)^{\frac{1}{h\,(p-1)}}.
\]
This concludes the proof.
\end{proof}

\begin{prop}[Uniform Lipschitz estimate]
\label{prop:a_priori_estimate}
Let $p\ge 2$, $h>N/2$ and $0<\varepsilon\le \varepsilon_0$. For every $B_{r_0}\subset B_{R_0}\Subset B$ with \(0<r_0<R_0\le 1\),  we have
\begin{equation}
\label{lipschitzf}
\|\nabla u_\varepsilon\|_{L^\infty(B_{r_0})}\le C\,\left(\frac{1+\|\nabla f_\varepsilon\|^{\sigma_2}_{L^{h}(B_{R_0})}}{(R_0-r_0)^{\sigma_2}}\right)\, \Big(\|\nabla u_\varepsilon\|^{\sigma_1}_{L^{p}(B_{R_0})}+1\Big),
\end{equation}
where $C=C(N,p,h, \delta)>1$ and $\sigma_i=\sigma_i(N,p,h)>0$, for $i=1,2$.
\end{prop}

\subsection{Proof of Proposition \ref{prop:a_priori_estimate}: the homogeneous case}
In this subsection, we assume that \(f=0\) and \(\delta=1\).

For simplicity, we assume throughout the proof that $N\ge 3$, so in this case the Sobolev exponent $2^*$ is { finite}. The case $N=2$ can be treated with minor modifications { and is left to the reader}. For ease of readability, we divide the proof into four steps. 
\vskip.2cm\noindent
{\bf Step 1: a first iterative scheme}. 
We add on both sides of inequality \eqref{russiancircles} the term
\[
\int |\nabla \eta|^2\, |u_{x_k}|^{2\,q+p}\,dx.
\]
We thus obtain
\[
\begin{split}
\int \left|\nabla \left(\left(|u_{x_k}|^{q+\frac{p-2}{2}}\,u_{x_k}\right)\,\eta\right)\right|^2\,dx
&\le C\,q^5\,\sum_{i,j=1}^N\int g_{i,\varepsilon}''(u_{x_i})\,|u_{x_j}|^{2\,q+2}\,|\nabla \eta|^2\,dx\\
&+C\,q^5\,\sum_{i=1}^N\int g_{i,\varepsilon}''(u_{x_i})\,|u_{x_k}|^{2\,q+2}\,|\nabla \eta|^2\,dx\\
&+C\, \int |\nabla \eta|^2\, |u_{x_k}|^{2\,q+p}\,dx.
\end{split}
\]
An application of Sobolev inequality leads to
\[
\begin{split}
\left(\int |u_{x_k}|^{\frac{2^*}{2}(2\,q+p)}\,\eta^{2^*}\,dx\right)^\frac{2}{2^*} &\le C\,q^5\,\sum_{i,j=1}^N\int g_{i,\varepsilon}''(u_{x_i})\,|u_{x_j}|^{2\,q+2}\,|\nabla \eta|^2\,dx\\
&+C\,q^5\,\sum_{i=1}^N\int g_{i,\varepsilon}''(u_{x_i})\,|u_{x_k}|^{2\,q+2}\,|\nabla \eta|^2\,dx\\
&+C\, \int |\nabla \eta|^2\, |u_{x_k}|^{2\,q+p}\,dx.
\end{split}
\]
We now sum over $k=1,\dots,N$ and use that by Minkowski inequality,
\[
\sum_{k=1}^N\left(\int |u_{x_k}|^{\frac{2^*}{2}(2\,q+p)}\,\eta^{2^*}\,dx\right)^\frac{2}{2^*} = \sum_{k=1}^N \left\||u_{x_k}|^{2\,q+p}\eta^2\right\|_{L^{\frac{2^*}{2}}} \geq  \left\|\sum_{k=1}^N|u_{x_k}|^{2\,q+p}\eta^2\right\|_{L^{\frac{2^*}{2}}}.
\]
This implies
\begin{equation}
\label{pronti??}
\begin{split}
\left( \int \left|\sum_{k=1}^N |u_{x_k}|^{2\,q+p}\right|^{\frac{2^*}{2}}\,\eta^{2^*}\,dx\right)^{\frac{2}{2^*}}
&\leq C\,q^5 \sum_{i, k=1}^{N} \int g_{i,\varepsilon}''(u_{x_i})\,|u_{x_k}|^{2\,q+2}\, |\nabla \eta|^2\,dx\\
&+ C\, \int |\nabla \eta|^2\, \sum_{k=1}^{N}|u_{x_k}|^{2\,q+p} \,dx.
\end{split}
\end{equation}
We now introduce the function 
\[
\mathcal{U}(x):= \max_{k=1, \dots, N}|u_{x_k}(x)|.
\]
We use that 
\[
\mathcal{U}^{2\,q+p}\leq \sum_{k=1}^{N}|u_{x_k}|^{2\,q+p} \leq N\, \mathcal{U}^{2\,q+p},
\]
and also that \(g_{i,\varepsilon}''(u_{x_i})\, |u_{x_k}|^{2\,q+2}\leq C\,\mathcal{U}^{2\,q+p}+\varepsilon\,\mathcal{U}^{2\,q+2}\) for every \(1\leq i, k \leq N\).
This yields
\[
\left( \int \mathcal{U}^{\frac{2^*}{2}(2\,q+p)}\, \eta^{2^*}\right)^{\frac{2}{2^*}} \leq C\,q^5 \int \mathcal{U}^{2\,q+p}|\nabla \eta|^2\,dx + Cq^5\varepsilon \int \mathcal{U}^{2q+2}\, |\nabla \eta|^2 \, dx
\]
for a possibly different $C=C(N,p)>1$. By using that \(\mathcal{U}^{2\,q+2}\leq 1 +\mathcal{U}^{2\,q+p}\), we obtain (for \(\varepsilon <1\))
\begin{equation}
\label{bo}
\left(\int \mathcal{U}^{\frac{2^*}{2}(2\,q+p)}\,\eta^{2^*}\,dx\right)^\frac{2}{2^*} 
\le C\, q^5 \,\int |\nabla \eta|^2\,\Big(\mathcal{U}^{2q+p}+1\Big)\,dx.
\end{equation}
We fix two concentric balls \(B_r\subset B_R \Subset B\)  and $0<r<R\le 1$.  
Let us assume for simplicity that all the balls are centered at the origin.  Then for every pair of radius $r\le t<s\le R$ we take in \eqref{bo} { a standard cut-off function
\begin{equation}
\label{eq_def_eta}
\eta\in C^\infty_0(B_s),\quad \eta\equiv 1\mbox{ on } B_t,\quad 0\le \eta\le 1,\quad \|\nabla \eta\|_{L^\infty}\le\frac{C}{s-t}.
\end{equation}}
This yields
\begin{equation}
\label{pronti!}
\left(\int_{B_t} \mathcal{U}^{\frac{2^*}{2}(2\,q+p)}\,dx\right)^\frac{2}{2^*} 
\le C\, \frac{q^5}{(s-t)^2} \,\int_{B_s} \Big(\mathcal{U}^{2\,q+p}+1\Big)\,dx.
\end{equation}

We define the sequence of exponents
\[
\gamma_j=p+2^{j+2}-2,\qquad j\in\mathbb{N},
\]
and take in \eqref{pronti!} $q=2^{j+1}-1$. This gives
\begin{equation}
\label{pronti!!}
\begin{split}
\left(\int_{B_{t}} \mathcal{U}^{\frac{2^*}{2}\gamma_j}\,dx\right)^\frac{2}{2^*}\le C\,\frac{2^{5\,j}}{(s-t)^2}\,\int_{B_{s}}\Big(\mathcal{U}^{\gamma_j}+1\Big)\,dx , 
\end{split}
\end{equation}
for a possibly different constant $C=C(N,p)>1$.
\vskip.2cm\noindent
{\bf Step 2: filling the gaps.}  We now observe that
\[
\gamma_{j-1}<\gamma_j<\frac{2^*}{2}\,\gamma_j,\qquad \mbox{ for every } j\in\mathbb{N}\setminus\{0\}.
\]
By interpolation in Lebesgue spaces, we obtain
\[
\int_{B_{t}} \mathcal{U}^{\gamma_j}\,dx\le \left(\int_{B_{t}} \mathcal{U}^{\gamma_{j-1}}\,dx\right)^\frac{\tau_j\,\gamma_j}{\gamma_{j-1}}\,\left(\int_{B_{t}} \mathcal{U}^{\frac{2^*}{2}\,\gamma_{j}}\,dx\right)^\frac{(1-\tau_j)\,2}{2^*}
\]
where $0<\tau_j<1$ is given by
\[
\tau_j=\frac{\frac{2^*}{2}-1}{\frac{2^*}{2}\,\dfrac{\gamma_j}{\gamma_{j-1}}-1}.
\]
We now rely on \eqref{pronti!!} to get
\[
\begin{split}
\int_{B_{t}} \mathcal{U}^{\gamma_j}\,dx&\le \left(\int_{B_{t}} \mathcal{U}^{\gamma_{j-1}}\,dx\right)^\frac{\tau_j\,\gamma_j}{\gamma_{j-1}}\, \left(C\,\frac{2^{5\,j}}{(s-t)^2}\,\int_{B_{s}}\Big(\mathcal{U}^{\gamma_j}+1\Big)\,dx\right)^{1-\tau_j}\\
&=\left[\left(C\,\frac{2^{5\,j}}{(s-t)^2}\right)^{\frac{1-\tau_j}{\tau_j}}\,\left(\int_{B_{t}} \mathcal{U}^{\gamma_{j-1}}\,dx\right)^\frac{\gamma_j}{\gamma_{j-1}}\right]^{\tau_j}\, \left(\int_{B_{s}}\Big(\mathcal{U}^{\gamma_j}+1\Big)\,dx\right)^{1-\tau_j}.
\end{split}
\]
The sequence \((\tau_j)_{j\geq 1}\) is decreasing, which implies
\[
\tau_j> \lim_{n\to\infty} \tau_n =\frac{1}{2}\frac{2^*-2}{2^*-1}=:\underline{\tau}\qquad \mbox{ for every } j\in\mathbb{N}\setminus\{0\}.
\]
Hence,
\[
\frac{1-\tau_j}{\tau_j} \leq \frac{1-\underline{\tau}}{\underline{\tau}}=:\beta.
\]
Using that \(s\leq R\le 1\) and $C>1$, this implies that
\[
\left(C\,\frac{2^{5\,j}}{(s-t)^2}\right)^{\frac{1-\tau_j}{\tau_j}} \leq \left(C\,\frac{2^{5\,j}}{(s-t)^2}\right)^{\beta}.
\]
By Young's inequality,
\[
\begin{split}
\int_{B_{t}} \mathcal{U}^{\gamma_j}\,dx&\le (1-\tau_j)\,\int_{B_{s}}\Big( \mathcal{U}^{\gamma_j}+1\Big)\,dx + \tau_j\,\left(C\,\frac{2^{5\,j}}{(s-t)^2}\right)^{\beta}\,\left(\int_{B_{t}} \mathcal{U}^{\gamma_{j-1}}\,dx\right)^\frac{\gamma_j}{\gamma_{j-1}}\\
&\le (1-\underline{\tau})\,\int_{B_{s}}\mathcal{U}^{\gamma_j}\,dx + C\,\frac{2^{5\,j\,\beta}}{(s-t)^{2\,\beta}}\,\left(\int_{{B_R}} \mathcal{U}^{\gamma_{j-1}}\,dx\right)^\frac{\gamma_j}{\gamma_{j-1}} + |B_R|.
\end{split}
\]
By applying Lemma \ref{lm:giusti} with
\[
Z(t)= \int_{B_{t}} \mathcal{U}^{\gamma_j}\,dx ,\qquad \alpha_0=2\, \beta, \qquad \mbox{ and }\qquad \vartheta=1-\underline{\tau},
\]
we finally obtain
\begin{equation}
\label{conj}
\int_{B_r} \mathcal{U}^{\gamma_j}\,dx\le C\,\left( 2^{5\,j\,\beta}\,(R-r)^{-2\,\beta}\,\left(\int_{B_R} \mathcal{U}^{\gamma_{j-1}}\,dx\right)^\frac{\gamma_j}{\gamma_{j-1}}+ 1\right),
\end{equation}
for some $C=C(N,p)>1$.
\vskip.2cm\noindent
{\bf Step 3: Moser's iteration.} We now want to iterate the previous estimate on a sequence of shrinking balls. We fix two radii $0<r<R\le 1$, then we consider the sequence 
\[
R_j=r+\frac{R-r}{2^{j-1}},\qquad j\in\mathbb{N}\setminus\{0\},
\]
and we apply \eqref{conj} with \(R_{j+1}<R_j\) instead of \(r<R\). 
Thus we get 
\begin{equation}
\label{scamone}
\int_{B_{R_{j+1}}} \mathcal{U}^{\gamma_j}\,dx
\le \,C\,\left(2^{7\,j\,\beta}\,(R-r)^{-2\,\beta}\left( \int_{B_{R_j}}\mathcal{U}^{\gamma_{j-1}}\,dx \right)^{\frac{\gamma_j}{\gamma_{j-1}}}+ 1\right)
\end{equation}
where the constant \(C>1\) depends on $N$ and $p$ only.
\par
We introduce the notation
\[
Y_j=\int_{B_{R_{j}}} \mathcal{U}^{\gamma_{j-1}}\,dx,
\]
thus \eqref{scamone} rewrites as
\[
Y_{j+1} \le \,C\,\left(2^{7\,j\,\beta}\,(R-r)^{-2\,\beta}\,Y_{j}^{\frac{\gamma_j}{\gamma_{j-1}}}+ 1\right)
\le {2}\,C\,2^{7\,j\,\beta}\,(R-r)^{-2\,\beta}\,(Y_{j}+1)^{\frac{\gamma_j}{\gamma_{j-1}}}.
\]
Here, we have used again that $R\le 1$, { so that the term multiplying $Y_j$ is larger than $1$.}
By iterating the previous estimate starting from $j=1$ { and using some standard manipulations}, we obtain
\[
\begin{split}
Y_{n+1}&\le \Big(C\,2^{7\,\beta}\,(R-r)^{-2\,\beta}\Big)^{\sum\limits_{j=0}^{n-1}(n-j)\frac{\gamma_n}{\gamma_{n-j}}}\,\Big[Y_1+1\Big]^\frac{\gamma_n}{\gamma_0},
\end{split}
\]
{ possibly for a different constant $C=C(N,p)>1$.}
We now take the power $1/\gamma_n$ on both sides:
\[
Y_{n+1}^\frac{1}{\gamma_n}\le \Big(C\,2^{7\,\beta}\,(R-r)^{-2\,\beta}\Big)^{\sum\limits_{j=0}^{n-1}\frac{n-j}{\gamma_{n-j}}}
\,\Big[Y_1+1\Big]^\frac{1}{\gamma_0}=\Big(C\,2^{7\, \beta}\,(R-r)^{-2\, \beta}\Big)^{\sum\limits_{j=1}^{n}\frac{j}{\gamma_{j}}}
\,\Big[Y_1+1\Big]^\frac{1}{\gamma_0}.
\]
We observe that \(\gamma_{j}\sim 2^{j+2} \) as \(j\) goes to \(\infty\). This implies the convergence of the series above and we thus get
\[
\|\mathcal{U}\|_{L^{\infty}(B_{r})} =  \lim_{n\to\infty}\left(\int_{B_{R_{n+1}}} \mathcal{U}^{\gamma_{n+1}}\,dx\right)^\frac{1}{\gamma_{n+1}} \leq C\,  (R-r)^{-\beta'}\,\left(\int_{B_{R}} \mathcal{U}^{p+2}\,dx+1\right)^\frac{1}{p+2},
\]
for some $C=C(N,p)>{ 1}$ and $\beta'=\beta'(N,p)>0$. We also used that $\gamma_0=p+2$. 
By recalling the definition of \(\mathcal{U}\), we finally obtain
\begin{equation}
\label{lipschitz2}
\|\nabla u\|_{L^{\infty}(B_{r})} \leq C\,(R-r)^{-\beta'}\, \left(\int_{B_{R}} |\nabla u|^{p+2}\,dx+1\right)^{\frac{1}{p+2}}.
\end{equation}
\vskip.2cm\noindent
{\bf Step 4: $L^\infty-L^p$ estimate}
We fix two concentric balls $B_{r_0}\subset B_{R_0}\Subset B$ with $R_0\le 1$. Then for every $r_0\le t<s\le R_0$ from \eqref{lipschitz2} we have
\[
\|\nabla u\|_{L^\infty(B_{t})}\le \frac{C}{(s-t)^{\beta'}}\, \left(\int_{B_{s}}|\nabla u|^{p+2}\,dx\right)^\frac{1}{p+2}+\frac{C}{(s-t)^{\beta'}},
\]
where we also used the subadditivity of $\tau\mapsto \tau^{1/(p+2)}$. We now observe that 
\[
\begin{split}
\frac{C}{(s-t)^{\beta'}}\, \left(\int_{B_{s}}|\nabla u|^{p+2}\,dx\right)^\frac{1}{p+2}&\le \frac{C}{(s-t)^{\beta'}}\, \left( \int_{B_{s}}|\nabla u|^{p}\,dx\right)^\frac{1}{p+2}\,\|\nabla u\|_{L^\infty(B_{s})}^\frac{2}{p+2}\\
&\le \frac{2}{p+2}\,\|\nabla u\|_{L^\infty(B_{s})}\\
&+\frac{p}{p+2}\,\left(\frac{C}{(s-t)^{\beta'}}\right)^\frac{p+2}{p}\, \left( \int_{B_{s}}|\nabla u|^{p}\,dx\right)^\frac{1}{p}.
\end{split}
\]
We can apply again Lemma \ref{lm:giusti}, this time with the choices
\[
Z(t)=\|\nabla u\|_{L^\infty(B_{t})},\quad \mathcal{A}=\frac{p}{p+2}\,C^\frac{p+2}{p}\, \left(\int_{B_{R_0}}|\nabla u|^{p}\,dx\right)^\frac{1}{p},\quad \alpha_0=\frac{p+2}{p\,\beta'}{,\quad \beta_0=\beta'}.
\]
This yields
\[
\|\nabla u\|_{L^\infty(B_{r_0})}\le C\,\left[\frac{1}{(R_0-r_0)^{\beta'\,\frac{p+2}{p}}}\, \left(\int_{B_{R_0}}|\nabla u|^{p}\,dx\right)^\frac{1}{p}+\frac{1}{(R_0-r_0)^{\beta'}}\right],
\]
for every $R_0\le 1$. This readily implies the desired estimate \eqref{lipschitzf} in the homogeneous case.\hfill $\square$

\subsection{Proof of Proposition \ref{prop:a_priori_estimate}: the non-homogeneous case}
\label{subsec:ule2}

We follow step by step the proof of the homogeneous case and we only indicate the main changes, which essentially occur in {\bf Step 1} and {\bf Step 2}.
\vskip.2cm\noindent
{\bf Step 1: a first iterative scheme}.
This times, we add on both sides of inequality \eqref{russiancirclesf} the term
\[
\int |\nabla \eta|^2\,(|u_{x_k}|-\delta_k)_{+}^p\, |u_{x_k}|^{2\,q}\,dx.
\]
Then the left-hand side { is greater, up to a constant, than}
\[
\int \left|\nabla \left( (|u_{x_k}|-\delta_k)_{+}^{\frac{p}{2}}\,|u_{x_k}|^q\,\eta\right)\right|^2\,dx.
\]
{ The latter in turn}, by Sobolev inequality is greater, up to a constant, than  
\[
\left(\int (|u_{x_k}|-\delta_k)_{+}^{\frac{2^*\,p}{2}}\, |u_{x_k}|^{2^*q}\,\eta^{2^*}\,dx\right)^\frac{2}{2^*}. 
\]
By summing over $k=1,\dots,N$ and using Minkowski inequality, we obtain the analogue of \eqref{pronti??}, namely
\[
\begin{split}
\left( \int \Big|\sum_{k=1}^N(|u_{x_k}|-\delta_k)_{+}^{p}\,|u_{x_k}|^{2\,q}\Big|^{\frac{2^*}{2}}\eta^{2^*}\,dx\right)^{\frac{2}{2^*}}
&\leq Cq^5 \sum_{i, k=1}^{N} \int g_{i,\varepsilon}''(u_{x_i})\,|u_{x_k}|^{2q+2}\, |\nabla \eta|^2\,dx\\
&+C\,q^5\, \sum_{k=1}^{N}\int |\nabla f|\, |u_{x_k}|^{2\,q+1}\,\eta^2\,dx\\
&+ C \int |\nabla \eta|^2 \sum_{k=1}^{N} (|u_{x_k}|-\delta_k)_{+}^p\, |u_{x_k}|^{2\,q}\,dx.
\end{split}
\]
We now introduce the  function 
\[
\mathcal{U}(x):= \frac{1}{2\,\delta}\max_{k=1, \dots, N}|u_{x_k}(x)|,
\]
where the parameter \(\delta\) is defined in \eqref{delta}.
We use that 
\[
\sum_{k=1}^N(|u_{x_k}|-\delta_k)_{+}^{p}\,|u_{x_k}|^{2\,q}\geq (2\,\delta\,\mathcal{U}-\delta)_{+}^{p}\,|2\,\delta\,\mathcal{U}|^{2\,q}
\geq (2\,\delta)^{2\,q+p}\, \left(\mathcal{U}-\frac{1}{2}\right)_{+}^p\,\mathcal{U}^{2\,q},
\]
and also that for every $1\leq i \leq N$,
\[
g_{i,\varepsilon}''(u_{x_i})=(p-1)\,(|u_{x_i}|-\delta_i)_{+}^{p-2}+\varepsilon\le C\,\delta^{p-2}\,\mathcal{U}^{p-2}+\varepsilon.
\] 
This yields
\[
\begin{split}
\left( \int  \left(\mathcal{U}-\frac{1}{2}\right)_{+}^{\frac{2^*}{2}\,p} \mathcal{U}^{2^* q}\, \eta^{2^*}\,dx\right)^{\frac{2}{2^*}} &\leq C\,q^5 \int \mathcal{U}^{2\,q+p}\,|\nabla \eta|^2\,dx + C\,q^5\varepsilon \int \mathcal{U}^{2\,q+2}\, |\nabla \eta|^2 \, dx  \\
&\qquad +C\,q^5\, \int |\nabla f|\, \mathcal{U}^{2\,q+1}\,\eta^2\,dx 
\end{split}
\]
for a possibly different $C=C(N,p, \delta)>1$. 
\par
With the {concentric} balls \(B_r\subset B_t \subset B_s \subset B_R\)  and the function \(\eta\) as defined in \eqref{eq_def_eta}, 
an application of H\"older's inequality leads to
\begin{equation}
\label{pronti!!;f}
\begin{split}
\left( \int_{B_t}  \left(\mathcal{U}-\frac{1}{2}\right)_{+}^{\frac{2^*}{2}\,p} \mathcal{U}^{2^*\,q} \,dx\right)^{\frac{2}{2^*}} &\leq C\,\frac{q^5}{(s-t)^2} \int_{B_s} \mathcal{U}^{2\,q+p}\,dx + C\,\frac{q^5}{(s-t)^2}\,\varepsilon\, \int_{B_s} \mathcal{U}^{2\,q+2}\, dx \\
& +C\,q^5\, \|\nabla f\|_{L^h(B_R)}\,\left(\int_{B_s} \mathcal{U}^{(2\,q+1)\,h'}\,dx\right)^\frac{1}{h'}.
\end{split}
\end{equation}
From now on, we assume that
\begin{equation}
\label{limit_q}
q\ge \max\left\{\frac{p-2\,h'}{2\,(h'-1)},\, \frac{2^*\,p}{{2}\,h'}-1\right\}.
\end{equation} 
This in particular implies that 
\[
2\,q+2\le 2\,q+p\le (2\,q+2)\,h',
\]
then by using H\"older's inequality and taking into account that $s\le 1$, we get
\[
\begin{split}
\left( \int_{B_t}  \left(\mathcal{U}-\frac{1}{2}\right)_{+}^{\frac{2^*}{2}\,p} \mathcal{U}^{2^*\,q} \,dx\right)^{\frac{2}{2^*}} &\leq C\,\frac{q^5}{(s-t)^2} \left(\int_{B_s} \mathcal{U}^{(2\,q+2)\,h'}\,dx\right)^\frac{2\,q+p}{(2\,q+2)\,h'}\\
&+ C\,\frac{q^5}{(s-t)^2}\,\varepsilon\, \left(\int_{B_s} \mathcal{U}^{(2\,q+2)\,h'}\, dx\right)^\frac{1}{h'} \\
& +C\,q^5\, \|\nabla f\|_{L^h(B_R)}\,\left(\int_{B_s} \mathcal{U}^{(2\,q+2)\,h'}\,dx\right)^\frac{2\,q+1}{(2\,q+2)\,h'}.
\end{split}
\]
Thanks to the relation on the exponents, this gives (recall that $\varepsilon<1$ and $s\le 1$)
\begin{equation}
\label{pronti_bis_bis!}
\begin{split}
\left( \int_{B_t}  \left(\mathcal{U}-\frac{1}{2}\right)_{+}^{\frac{2^*}{2}\,p} \mathcal{U}^{2^*\,q} \,dx\right)^{\frac{2}{2^*}}&\leq \frac{C\,q^5}{(s-t)^2}\left(1 + \|\nabla f\|_{L^{h}(B_{R})}\right)\\
&\times \left(\int_{B_s} \mathcal{U}^{(2\,q+2)\,h'}\,dx+1\right)^{\frac{2\,q+p}{(2\,q+2)\,h'}}. 
\end{split}
\end{equation}
We now estimate
\[
\begin{split}
\int_{B_s} \mathcal{U}^{(2\,q+2)\,h'}\,dx &= \int_{B_s\cap \{\mathcal{U}\geq 1\}} \mathcal{U}^{(2\,q+2)\,h'}\,dx + \int_{B_s\cap \{\mathcal{U}\leq 1\}}\mathcal{U}^{(2\,q+2)\,h'}\,dx\\
&  \leq
\int_{B_s\cap \{\mathcal{U}\geq 1\}} \mathcal{U}^{(2\,q+2)\,h'}\,dx +C.
\end{split}
\]
Observe that on the set \(\{\mathcal{U}\geq 1\}\), we have \(\mathcal{U}\leq 2\,\left(\mathcal{U}-1/2\right)_+\). Hence,
\begin{equation}
\label{anna}
\int_{B_s} \mathcal{U}^{(2\,q+2)\,h'}\,dx \leq  C\,\int_{B_s} \left(\mathcal{U}-\frac{1}{2} \right)_{+}^{\frac{2^*}{2}\,p} \mathcal{U}^{(2\,q+2)\,h'-\frac{2^*}{2}\,p}\,dx  +C,  
\end{equation}
{ where the exponent $(2\,q+2)\,h'-(2^*p)/2$ is positive, thanks to the choice \eqref{limit_q} of $q$.}
We deduce from \eqref{pronti_bis_bis!} that
\begin{equation}
\label{pronti_bis_bis!!}
\begin{split}
\left(\int_{B_t}  \left(\mathcal{U}-\frac{1}{2}\right)_{+}^{\frac{2^*}{2}\,p} \mathcal{U}^{2^*\,q} \,dx\right)^{\frac{2}{2^*}} &\leq \frac{C\,q^5}{(s-t)^2}\left(1 + \|\nabla f\|_{L^{h}(B_{R})}\right)\\
&\times \left(\int_{B_s} \left(\mathcal{U}-\frac{1}{2} \right)_{+}^{\frac{2^*}{2}\,p}\,\mathcal{U}^{(2\,q+2)\,h'-\frac{2^*}{2}\,p}\,dx+1\right)^{\frac{2\,q+p}{(2\,q+2)\,h'}},
\end{split}
\end{equation}
for a constant $C=C(N,p,h,\delta)>1$.
We now take $q=2^{j+1}-1$ for $j\ge j_0-1$, where $j_0\in\mathbb{N}$ is chosen so as to ensure condition \eqref{limit_q}.
Then we define the sequence of positive exponents
\[
\gamma_j=(2\,q+2)\,h'-\frac{2^*}{2}\,p=2^{j+2}\,h'-\frac{2^*}{2}\,p,\qquad j\ge j_0,
\]
and
\[
\widehat{\gamma}_j=2^*\,q=2^*\,(2^{j+1}-1),\qquad j\ge j_0.
\]
In order to simplify the notation, we also introduce the { absolutely continuous} measure
\[
d\,\mu:=\left(\mathcal{U}-\frac{1}{2}\right)_{+}^{\frac{2^*}{2}\,p}\,dx.
\]
From \eqref{pronti_bis_bis!!}, we get
\[
\left( \int_{B_t}  \mathcal{U}^{\widehat{\gamma}_j} \,d\mu\right)^{\frac{2}{2^*}} \leq \frac{C\,2^{5\,j}}{(s-t)^2}\left(1 + \|\nabla f\|_{L^{h}(B_{R})}\right)
\, \left(\int_{B_s} \mathcal{U}^{\gamma_j}\,d\mu+1\right)^{\frac{2}{2^*}\,\frac{\widehat \gamma_j+\frac{2^*}{2}\,p}{\gamma_j+\frac{2^*}{2}\,p}}. 
\]
We now observe that { $h>N/2$ implies $h'<2^*/2$. By recalling that $p\ge 2$, we thus have} \(2\,h'<(2^*\,p)/2\), which in turn implies
\begin{equation}
\label{eq_ratio1}
\frac{\widehat{\gamma}_j}{\gamma_j}\ge \frac{2^*}{2\,h'}> 1, \qquad j\ge j_0.
\end{equation}
It follows that
\[
\frac{\widehat \gamma_j+\dfrac{2^*}{2}\,p}{\gamma_j+\dfrac{2^*}{2}\,p} \leq \frac{\widehat \gamma_j}{\gamma_j}. 
\]
Hence, we obtain
\begin{equation}
\label{pronti!!f}
\left( \int_{B_t} \mathcal{U}^{\widehat{\gamma}_j} \,d\mu\right)^{\frac{2}{2^*}} \leq \frac{C\,2^{5\,j}}{(s-t)^2}\left(1 + \|\nabla f\|_{L^{h}(B_{R})}\right)\,\left(\int_{B_s} \mathcal{U}^{\gamma_j}\,d\mu+1\right)^{\frac{2}{2^*}\,\frac{\widehat \gamma_j}{\gamma_j}}. 
\end{equation}
\noindent
{\bf Step 2: filling the gaps.} Since
\[
\gamma_{j-1}<\gamma_j<\widehat \gamma_j,\qquad \mbox{ for every }j\ge {j_0+1},
\]
 we obtain by interpolation in Lebesgue spaces,
\[
\int_{B_{t}} \mathcal{U}^{\gamma_j}\,d\mu \le \left(\int_{B_{t}} \mathcal{U}^{\gamma_{j-1}}\,d\mu\right)^\frac{\tau_j\,\gamma_j}{\gamma_{j-1}}\,\left(\int_{B_{t}}  \mathcal{U}^{\widehat\gamma_j}\,d\mu\right)^\frac{(1-\tau_j)\,\gamma_j}{\widehat \gamma_j},
\]
where $0<\tau_j<1$ is given by
\begin{equation}
\label{taujf}
\tau_j=\frac{\dfrac{\widehat{\gamma}_j}{\gamma_j}-1}{\dfrac{\widehat{\gamma}_j}{\gamma_j}\,\dfrac{\gamma_j}{\gamma_{j-1}}-1}.
\end{equation}
We now rely on \eqref{pronti!!f} to get
\begin{equation}
\label{eq_ready}
\begin{split}
\int_{B_{t}} &\mathcal{U}^{\gamma_j}\,d\mu \le \left(\int_{B_{t}} \mathcal{U}^{\gamma_{j-1}}\,d\mu\right)^\frac{\tau_j\,\gamma_j}{\gamma_{j-1}}\\
&\qquad \times \left[\left(C\,\frac{2^{5\,j}}{(s-t)^2}(1+ \|\nabla f\|_{L^{h}(B_{R})})\right)^{\frac{2^*\,\gamma_j}{2\,\widehat{\gamma}_j}}\,\left(\int_{B_{s}} \mathcal{U}^{\gamma_j}\,d\mu+1\right)\right]^{1-\tau_j}\\
&=\left[\left(C\,\frac{2^{5\,j}}{(s-t)^2}(1+\|\nabla f\|_{L^{h}(B_{R})})\right)^{\frac{2^*\,\gamma_j\,(1-\tau_j)}{2\,\widehat{\gamma}_j\,\tau_j}}\,\left(\int_{B_{t}} \mathcal{U}^{\gamma_{j-1}}\,d\mu\right)^\frac{\gamma_j}{\gamma_{j-1}}\right]^{\tau_j}\\
&\qquad\times \left(\int_{B_{s}}  \mathcal{U}^{\gamma_j}\,d\mu+1\right)^{1-\tau_j}.
\end{split}
\end{equation}
We claim that 
\begin{equation}
\label{eq_def_tau}
\tau_j\geq \underline{\tau}:=\frac{2^*-2\,h'}{4\cdot 2^*-2\,h'}\qquad \mbox{ for every } j\ge {j_0+1}.
\end{equation}
We already know by \eqref{eq_ratio1} that \((\widehat{\gamma}_j/\gamma_j) \geq 2^*/(2h')\).
Moreover, relying on the fact that \((2^*\,p)/2\leq 2^{j_0}\,h'\) (this follows from the definition of $j_0$), we also have
\[
2\le \frac{\gamma_j}{\gamma_{j-1}} \le  4,\qquad j\ge {j_0+1}.
\]
By recalling the definition \eqref{taujf} of $\tau_j$, we get
\[
\tau_j=\zeta\left(\frac{\widehat{\gamma}_j}{\gamma_j}, \frac{\gamma_j}{\gamma_{j-1}}\right),\qquad \mbox{ where } \zeta(x,y) = \frac{x-1}{x\,y-1}.
\]
Observe that on $[2^*/(2\,h'),+\infty)\times [2,4]$, the function \(x\mapsto \zeta(x,y)\) is increasing, while $y\mapsto \zeta(x,y)$ is decreasing. Thus we get
\[
\tau_j\ge \zeta\left(\frac{2^*}{2\,h'},4\right),
\]
which is exactly claim \eqref{eq_def_tau}.
We deduce from \eqref{eq_def_tau} and \eqref{eq_ratio1} that
\[
\frac{2^*\,\gamma_j\,(1-\tau_j)}{2\,\widehat{\gamma}_j\,\tau_j}\leq \frac{1-\underline{\tau}}{\underline{\tau}}\,h'=:\beta.
\]
In particular, we have
\[
\left(C\,\frac{2^{5\,j}}{(s-t)^2}(1+\|\nabla f\|_{L^{h}(B_{R})})\right)^{\frac{2^*\,\gamma_j\,(1-\tau_j)}{2\,\widehat{\gamma}_j\,\tau_j}}\leq \left(C\,\frac{2^{5\,j}}{(s-t)^2}(1+\|\nabla f\|_{L^{h}(B_{R})})\right)^{\beta},
\]
since the quantity inside the parenthesis is larger than \(1\) (here, we use again that \(s\leq 1\)).
In view of \eqref{eq_ready}, this implies
\[
\begin{split}
\int_{B_{t}} \mathcal{U}^{\gamma_j}\,d\mu&\leq \left[\left(C\,\frac{2^{5\,j}}{(s-t)^2}(1+\|\nabla f\|_{L^{h}(B_{R})})\right)^{\beta}\,\left(\int_{B_{t}} \mathcal{U}^{\gamma_{j-1}}\,d\mu\right)^\frac{\gamma_j}{\gamma_{j-1}}\right]^{\tau_j}\\
&\times \left(\int_{B_{s}} \mathcal{U}^{\gamma_j}\,d\mu+1\right)^{1-\tau_j}.
\end{split}
\]
By Young's inequality,
\[
\begin{split}
\int_{B_{t}} \mathcal{U}^{\gamma_j}\,d\mu &\le (1-\tau_j)\,\left(\int_{B_{s}} \mathcal{U}^{\gamma_j}\,d\mu+1\right)\\
&+ \tau_j\,\left(C\,\frac{2^{5\,j}}{(s-t)^2}\,(1+ \|\nabla f\|_{L^{h}(B_{R})})\right)^{\beta}\,\left(\int_{B_{t}} \mathcal{U}^{\gamma_{j-1}}\,d\mu\right)^\frac{\gamma_j}{\gamma_{j-1}}\\
&\le (1-\underline{\tau})\,\int_{B_{s}}\mathcal{U}^{\gamma_j}\,d\mu \\
&+ C\,\frac{2^{5\,j\,\beta}}{(s-t)^{2\,\beta}}\,(1+\|\nabla f\|_{L^{h}(B_{R})})^{\beta}\,\left(\int_{B_{R}} \mathcal{U}^{\gamma_{j-1}}\,d\mu\right)^\frac{\gamma_j}{\gamma_{j-1}} + 1,
\end{split}
\]
where \(C=C(N,p,h,\delta)>1\) as usual.
By applying again Lemma \ref{lm:giusti}, this times with the choices
\[
Z(t)= \int_{B_{t}} \mathcal{U}^{\gamma_j}\,d\mu ,\qquad \alpha_0=2\,\beta, \qquad \mbox{ and }\qquad \vartheta=1-\underline{\tau},
\]
we finally obtain
\begin{equation}
\label{conjf}
\int_{B_r} \mathcal{U}^{\gamma_j}\,d\mu \le C\,\frac{2^{5\,j\,\beta}}{(R-r)^{2\,\beta}}\,(1+\|\nabla f\|_{L^{h}(B_{R})})^{\beta}\,\left(\int_{B_{R}} \mathcal{U}^{\gamma_{j-1}}\,d\mu\right)^\frac{\gamma_j}{\gamma_{j-1}}+C.
\end{equation}

{\bf Step 3: Moser's iteration.}
Estimate \eqref{conjf} is the analogue of \eqref{conj}, except that the Lebesgue measure \(dx\) is now replaced by the measure \(d\mu\), and the index \(j\) is assumed to be larger than some \(j_0+1\), instead of \(j\geq 0\) as in \eqref{conj}. Following the same iteration argument { and starting from $j=j_0+1$}, we are led to 
\begin{equation}\label{eq39}
\|\mathcal{U}\|_{L^{\infty}(B_{r},\,d\mu)}\leq C\, \left(\frac{1+\|\nabla f\|_{L^{h}(B_{R})}}{R-r}\right)^{\beta'}\,\left(\int_{B_{R}} \mathcal{U}^{\gamma_{j_0}}\,d\mu+1\right)^\frac{1}{\gamma_{j_0}},
\end{equation}
for some $C=C(N,p,h, \delta)>1$, $\beta'=\beta'(N,p,h)>0$.
\vskip.2cm\noindent
{\bf Step 4: $L^{\infty}-L^{p}$ estimate.} We now  want to replace the norm {$L^{\gamma_{j_0}}(B_R,d\mu)$} of $\mathcal{U}$ in the right-hand side of \eqref{eq39} by its norm {$L^p(B_R,dx)$}. Let \(q_1:=2^{j_1+1}-1\) where 
\[
j_1:=\min \left\{j\geq j_0 : j+1\geq \log_2\left(1+\frac{\gamma_{j_0}}{2^*}\right)\right\}.
\]
Then \(\gamma_{j_0}\leq 2^*\,q_1\) and thus, by using that
\[
\mathcal{U}^{\gamma_{j_0}}\le 2^{2^*q_1-\gamma_{j_0}}\,\mathcal{U}^{2^*q_1},\qquad \mbox{ whenever } \mathcal{U}\ge \frac{1}{2},
\]
we have
\begin{equation}
\label{eq51}
\|\mathcal{U}\|_{L^{\gamma_{j_0}}(B_R,\,d\mu)}\leq C\,\|\mathcal{U}\|^\frac{2^*q_1}{\gamma_{j_0}}_{L^{2^* q_1}(B_R,\,d\mu)}. 
\end{equation}
We  rely on \eqref{pronti_bis_bis!} with \(q=q_1\) to get for every \(0<r<t<s<R\)
\begin{equation}\label{eq52}
\|\mathcal{U}\|_{L^{2^* q_1}(B_t,\,d\mu)}^{2\,q_1} \leq \frac{C}{(s-t)^2}\,\left(1+\|\nabla f\|_{L^{h}(B_{R})}\right)\,\left(\|\mathcal{U}\|_{L^{2\,(q_1+1)\,h'}(B_s)}^{2\,q_1+p} +1\right),
\end{equation}
for some new constant \(C=C(N,p,h,\delta)>1\). 
\par
Since \(j_1\geq j_0\), we have \(p<(2\,q_1+2)\,h'<2^*/2\,(2\,q_1+p)\), and thus, by interpolation in Lebesgue spaces
\begin{equation}
\label{eq56}
\|\mathcal{U}\|_{L^{2\,(q_1+1)\,h'}(B_s)}\leq \|\mathcal{U}\|_{L^{2^*\,q_1+\frac{2^*}{2}\,p}(B_s)}^{\theta}\, \|\mathcal{U}\|_{L^{p}(B_s)}^{1-\theta},
\end{equation}
where \(\theta\in (0,1)\) is determined as usual by scale invariance. 
As in the proof of \eqref{anna}, we have
\[
\|\mathcal{U}\|_{L^{2^*q_1+\frac{2^*}{2}\,p}(B_s)}\leq C\,\|\mathcal{U}\|_{L^{2^*q_1}(B_s,\,d\mu)}^{\frac{2\,q_1}{2\,q_1+p}} + C.
\]
Inserting this last estimate into \eqref{eq56}, we obtain
\[
\|\mathcal{U}\|_{L^{2\,(q_1+1)\,h'}(B_s)}^{2\,q_1+p}\leq C\, \|\mathcal{U}\|_{L^{2^*q_1}(B_s,\, d\mu)}^{2\,q_1\,\theta} \|\mathcal{U}\|_{L^{p}(B_s)}^{(1-\theta)\,(2\,q_1+p)} + C\,\|\mathcal{U}\|_{L^{p}(B_s)}^{(1-\theta)\,(2\,q_1+p)},
\]
up to changing the constant $C=C(N,p,h,\delta)>1$.
In view of \eqref{eq52}, this gives
\[
\begin{split}
\|\mathcal{U}\|_{L^{2^*q_1}(B_t,\,d\mu)}^{2\,q_1}& \leq \frac{C}{(s-t)^2}\,\left(1+\|\nabla f\|_{L^{h}(B_{R})}\right)\\
&\times\left(\|\mathcal{U}\|_{L^{2^*q_1}(B_s,\,d\mu)}^{2\,q_1\,\theta}\, \|\mathcal{U}\|_{L^{p}(B_s)}^{(1-\theta)\,(2\,q_1+p)} + \|\mathcal{U}\|_{L^{p}(B_s)}^{(1-\theta)\,(2\,q_1+p)}+1
\right).
\end{split}
\] 
By Young's inequality, we get
\[
\begin{split}
\|\mathcal{U}\|_{L^{2^*q_1}(B_t,\,d\mu)}^{2\,q_1} &\leq \theta\,\|\mathcal{U}\|_{L^{2^*q_1}(B_s,\,d\mu)}^{2\,q_1} + (1-\theta)\,\left(\frac{C}{(s-t)^2}(1+\|\nabla f\|_{L^{h}(B_{R})})\right)^{\frac{1}{1-\theta}}
\|\mathcal{U}\|_{L^{p}(B_{R})}^{\,(2\,q_1+p)}\\
&\qquad + \frac{C}{(s-t)^2}\,\left(1+\|\nabla f\|_{L^{h}(B_{R})}\right)\, \left(\|\mathcal{U}\|_{L^{p}(B_{R})}^{(1-\theta)\,(2\,q_1+p)}+1\right).
\end{split}
\]
By Lemma \ref{lm:giusti}, this implies
\[
\|\mathcal{U}\|_{L^{2^*\,q_1}(B_r,\,d\mu)}^{2\,q_1} \leq C\,\left(\frac{1}{(R-r)^2}(1+\|\nabla f\|_{L^{h}(B_{R_0})})\right)^{\frac{1}{1-\theta}} \left(\|\mathcal{U}\|_{L^{p}(B_R)}^{\,(2\,q_1+p)} + 1\right),
\]
{ after some standard manipulations.}
Coming back to \eqref{eq39} and taking into account \eqref{eq51}, we obtain
\[
\|\mathcal{U}\|_{L^{\infty}(B_{r_0},\,d\mu)}\leq C\, \left(\frac{1+\|\nabla f\|_{L^{h}(B_{R_0})}}{R_0-r_0}\right)^{\sigma_2}\,\left(\|\mathcal{U}\|_{L^{p}(B_{R_0})}^{\sigma_1}+1\right),
\] 
where \(C=C(N,p,h,\delta)>1\) and \(\sigma_i=\sigma_i(N,p,h)>0\), for $i=1,2$.
By definition of \(\mathcal{U}\), we have 
\[
|\nabla u|\leq 2\,\delta\,\sqrt{N}\,\mathcal{U}\le \sqrt{N}\, |\nabla u|.
\] 
Since \(\|\mathcal{U}\|_{L^{\infty}(B_{r_0},\,d\mu)}+1\geq \|\mathcal{U}\|_{L^{\infty}(B_{r_0})}\), it  follows that
\[
\|\nabla u\|_{L^{\infty}(B_{r_0})} \leq C \, \left(\frac{1+\|\nabla f\|_{L^{h}(B_{R_0})}}{R_0-r_0}\right)^{\sigma_2}\,\left(\|\nabla u\|_{L^{p}(B_{R_0})}^{\sigma_1}+1\right),
\]
{ possibly for a different constant $C=C(N,p,h,\delta)>1$.}
This completes the proof.\hfill $\square$

\appendix
\section{Lipschitz regularity with a nonlinear lower order term}
\label{sec:nllot} 
In this section, we consider the functional
\[
\mathfrak{G}_\delta(u,\Omega')=\sum_{i=1}^N \int_{\Omega'} \Big[g_i(u_{x_i}) + G(x,u)\Big]\,dx ,\qquad \Omega'\Subset\Omega,\ u\in W^{1,p}_{\rm loc}(\Omega').
\]
The lower order term \(f\,u\) of the functional $\mathfrak{F}_\delta$ is thus replaced by a more general term \(G(x,u)\). We assume that \(G\) is a Carath\'eodory function and that for almost every \(x\in \Omega\), the map
\[
\xi \mapsto G(x,\xi) \qquad \textrm{is \(C^{1}\) and convex.}
\]
 We denote \(f(x,\xi):=G_\xi(x,\xi)\) and we assume that \(f\in W^{1,h}_{\rm loc}(\Omega \times \mathbb{R})\), for some \(h>N/2\).
Finally, we assume that \(G(x,\xi)\) satisfies the inequality
\begin{equation}\label{growth_G}
|G(x,\xi)|\leq b(x)\,|u|^{\gamma} +a(x)
\end{equation}
 where \(1<p\le \gamma <p^*\) and \(a, b\) are two non-negative functions belonging respectively to \(L^{s}_{\rm loc}(\Omega)\) and \(L^{\sigma}_{\rm loc}(\Omega)\) with \(s>N/p\) and \(\sigma>p^*/(p^*-\gamma)\).
 
Under assumption \eqref{growth_G}, all the local minimizers of $\mathfrak{G}_\delta$ are locally bounded, see \cite[Theorem 7.5]{Gi} and moreover, for every such minimizer \(u\), for every \(B_{r_0}\Subset B_{R_0} \Subset \Omega\),
\[
\|u\|_{L^{\infty}(B_{r_0})} \leq M,
\]
where \(M\) depends on \(\|u\|_{W^{1,p}(B_{R_0})}, r_0, R_0, \|b\|_{L^{\sigma}(R_0)}\), and \(\|a\|_{L^{s}(B_{R_0})}\).  
 \vskip.2cm\noindent
Then we have:
\begin{teo}
Let $p\ge 2$ and let ${U}\in W^{1,p}_{\rm loc}(\Omega)$ be a local minimizer of the functional $\mathfrak{G}_\delta$. Then ${U}$ is locally Lipschitz in $\Omega$. 
\end{teo}
\begin{proof}
We only explain the main differences with respect to the proof of Theorem \ref{teo:lipschitz}.
Since \(G\) is convex with respect to the second variable, the functional \(\mathfrak{G}\) is still convex. This implies that Lemma \ref{lm:convergence} remains true with the same proof. We then introduce the approximation of \(G\):
\[
G_{\varepsilon}(x,\xi) = \int_{\mathbb{R}^N\times \mathbb{R}}G(x-y, \xi-\zeta)\,\rho_{\varepsilon}(y)\,\widetilde{\rho}_{\varepsilon}(\zeta)\,dy\,d\zeta,
\]
where \(\rho_{\varepsilon}\) is the same regularization kernel as before, while \(\widetilde{\rho}_{\varepsilon}\) is a regularization kernel on \(\mathbb{R}\).

Given a local minimizer \(U\in W^{1,p}_{\rm loc}(\Omega)\) and a ball \(B\subset 2\,B\Subset \Omega\), there exists a unique \(C^{2}\) solution \(u_{\varepsilon}\) to the regularized problem 
\[
\min\left\{\mathfrak{G}_\varepsilon(v;B)\, :\, v-U_\varepsilon\in W^{1,p}_0(B)\right\},
\]
where 
\[
\mathfrak{G}_\varepsilon(v;B)=\sum_{i=1}^N \int_B g_{i,\varepsilon}(v_{x_i})\, dx+\int_B G_{\varepsilon}(x,v)\, dx
\]
and \(U_\varepsilon=U*\rho_{\varepsilon}\). Moreover, by \cite[Remark 7.6]{Gi} we have $u_\varepsilon \in L^\infty(B)$, with a bound on the $L^\infty$ norm uniform in $\varepsilon>0$. 
In order to simplify the notation, we simply write as usual \(u\) and \(f\) instead of \(u_{\varepsilon}\) and \(f_\varepsilon\).
The Euler equation is now
\[
\sum_{i=1}^N \int g'_{i,\varepsilon}(u_{x_i})\, \varphi_{x_i}\, dx+\int f(x,u)\, \varphi\, dx=0,\qquad \varphi\in W^{1,p}_0(B).
\]
When we differentiate the Euler equation with respect to some direction \(x_j\), we obtain 
\[
\sum_{i=1}^N \int g_{i,\varepsilon}''(u_{x_i})\, u_{x_i\,x_j}\, \psi_{x_i}\, dx{ +}\int \left( f_{x_j}(x,u)+f_{\xi}(x,u)\,u_{x_j}\right)\,\psi\,dx=0,\qquad \psi\in W^{1,p}_0(B).
\]
We can then repeat the proof of Proposition \ref{prop:a_priori_estimate} with this additional term \(f_{\xi}(x,u)u_{x_j}\) which leads to the following analogue of \eqref{pronti!!;f}:

\begin{equation*}
\begin{split}
\left( \int_{B_t}  \left(\mathcal{U}-\frac{1}{2}\right)_{+}^{\frac{2^*}{2}\,p} \mathcal{U}^{2^*\,q} \,dx\right)^{\frac{2}{2^*}} &\leq C\,\frac{q^5}{(s-t)^2} \int_{B_s} \mathcal{U}^{2\,q+p}\,dx + C\,\frac{q^5}{(s-t)^2}\,\varepsilon\, \int_{B_s} \mathcal{U}^{2\,q+2}\, dx \\
& +C\,q^5\, \|\nabla_x f\|_{L^h}\,\left(\int_{B_s} \mathcal{U}^{(2\,q+1)\,h'}\,dx\right)^\frac{1}{h'}\\
&+C\,q^5\, \|f_\xi\|_{L^h}\, \left(\int_{B_s}\mathcal{U}^{(2\,q+2)\,h'}\,dx\right)^{\frac{1}{h'}}.
\end{split}
\end{equation*}
Using again H\"older's inequality for the first three terms, we  obtain inequality \eqref{pronti_bis_bis!} where \(\|\nabla f\|\)  now represents the full gradient of \(f\) with respect to both \(x\) and \(\xi\). The rest of the proof is the same and leads to a uniform Lipschitz estimate, as desired.
\end{proof}


\end{document}